\documentclass[11pt,a4paper]{article}
\usepackage[margin=1.2in]{geometry}
\usepackage[T1]{fontenc}
\usepackage{tikz}
\usetikzlibrary{positioning}
\usepackage[utf8]{inputenc}
\usepackage{graphicx}
\usepackage{mathtools}
\usepackage{amssymb}
\usepackage{amsthm}
\usepackage{thmtools}
\usepackage{xcolor}
\usepackage{nameref}
\usepackage{hyperref}
\usepackage{amsmath}
\usepackage{amsfonts}
\usepackage{wrapfig}
\usepackage[autostyle]{csquotes}
\usepackage{layout}
\usepackage{tabularx}
\usepackage{authblk}
\usetikzlibrary{shapes.geometric, positioning}
\usepackage{caption}
\usepackage{mathrsfs}
\usepackage{subcaption}
\usetikzlibrary{positioning,chains,fit,shapes,calc}
\usepackage{tikz}
\theoremstyle{definition}
\newtheorem{defn}{Definition}[section]
\newtheorem{theorem}{Theorem}[section]
\newtheorem{corollary}{Corollary}[theorem]
\newtheorem{lemma}[theorem]{Lemma}

\theoremstyle{remark}
\newtheorem{remark}{Remark}[section]

\date{}
\begin{document}
		\title{\sc On Intersection and Co-maximal Hypergraph of $\mathbb{Z}_n$}
		
		\author{{ \sc Sachin Ballal}\footnote{Corresponding Author}~~  and \sc{ Ardra A N} }
	\affil{School of Mathematics and Statistics, University of Hyderabad,
		500046, India\\ sachinballal@uohyd.ac.in, 23mmpp02@uohyd.ac.in}

	\date{}
	\maketitle
	
	\begin{abstract}
 The aim of this paper is to study the intersection hypergraph  $\tilde{\Gamma}_\mathcal{H}(\mathbb{Z}_n)$ and co-maximal hypergraph $Co_\mathcal{H}(\mathbb{Z}_n)$ on the subgroups of $\mathbb{Z}_n$. We prove that  the intersection and co-maximal hypergraph of  a finite abelian group are isomorphic. Hence,  we focus on $\tilde{\Gamma}_\mathcal{H}(\mathbb{Z}_n)$ and examine some of the structural properties, viz., diameter, girth and chromatic number of $\tilde{\Gamma}_\mathcal{H}(\mathbb{Z}_n)$. Also, we provide characterizations for  hypertrees, star structures of  $\tilde{\Gamma}_\mathcal{H}(\mathbb{Z}_n)$, and investigate the planarity and non-planarity of  $\tilde{\Gamma}_\mathcal{H}(\mathbb{Z}_n)$.
\end{abstract}
	
	\section{Introduction}

		In \cite{berge}, C. Berge introduced the notion of hypergraphs as a natural generalisation of a graph which helps to study multiple relations over pair-wise relations. Hypergraphs on algebraic structures are a new area of study, generalizing graphs to extend key results from graph theory. Cyclic groups are one of the important class of groups and are the building blocks of abelian groups. Many researchers have analysed different types of graphs constructed from the subgroups of cyclic groups. In \cite{devicomplement}, P. Devi and R. Rajkumar studied the complement of the intersection graph of subgroups of a group. They classified all finite groups for which this graph is totally disconnected, bipartite, complete bipartite, a tree, a star graph, or 
		$C_3$-free. They also characterized all finite groups whose complement of the intersection graph is planar. In \cite{msaha}, M. Saha \textit{et al.} studied various structural properties of co-maximal subgroup graph of $\mathbb{Z}_n$.  In 
		particular, they characterized the value of $n$ for which the graph is hamiltonian, eulerian, perfect etc. In \cite{cameron2023hypergraphs}, P. J. Cameron \textit{et al.} have defined and studied various hypergraphs on elements of groups. In \cite{sb},  we have studied the intersection hypergraph on subgroups of a dihedral group. As subgroups play a key role in understanding groups, exploring hypergraphs on subgroups of a group can provide meaningful insights.

		A \textit{hypergraph} $\mathcal{H}$ is a pair $(V(\mathcal{H}),E(\mathcal{H}))$, where $V(\mathcal{H})$ is a set of  vertices and $E(\mathcal{H})$ is a set of hyperedges, where each hyperedge is a subset of $V(\mathcal{H})$. A hypergraph $\mathcal{H}'=(V'(\mathcal{H}'),E'(\mathcal{H}'))$   is called a \textit{subhypergraph} of  $\mathcal{H} = (V(\mathcal{H}),E(\mathcal{H}))$ if  $V'(\mathcal{H}')  \subseteq V(\mathcal{H})$ and $E'(\mathcal{H}') \subseteq E(\mathcal{H})$. A \textit{path in a hypergraph $\mathcal{H}$} is an alternating sequence of distinct
	vertices and edges of the form $v_1e_1v_2e_2...v_k$ such that $v_i,v_{i+1}$ is in $e_i$ for all $1 \leq i \leq k-1$. The \textit{cycle} 
	is a path whose first vertex is the same as the last vertex. The \textit{length} of a path is the
	number of hyperedges in the path. Two vertices $x_1$ and $x_2$ are said to be \textit{adjacent}, denoted by $x_1 \sim x_2$, if there exists a hyperedge containing both $x_1$ and $x_2$.  A hypergraph is said to be \textit{connected} if there exists a path between any two pair of vertices, otherwise 
	it is called a \textit{disconnected hypergraph}. The \textit{distance} between two vertices is the minimum length of the 
	path connecting these two vertices. The \textit{diameter} of a hypergraph is the maximum 
	distance among all pairs of vertices.  The \textit{girth} of a hypergraph is the length of a shortest cycle it contains. A hypergraph is called a \textit{star} if there is a vertex which belongs to all hyperedges.  
	The \textit{incidence graph (or bipartite representation)} $\mathcal{I}(\mathcal{H})$ of $\mathcal{H}$ is a bipartite graph with vertex set $V(\mathcal{H}) \cup E(\mathcal{H})$ and a vertex $v \in V(\mathcal{H}) $ is adjacent to a vertex $e \in E(\mathcal{H})$ if and only if $v \in e$ in $\mathcal{H}$. 	A \textit{proper vertex-coloring} (often simply called a proper coloring) of a hypergraph
	$\mathcal{H}$ is an assignment of colors to the vertices of $\mathcal{H}$ such that no hyperedge contains all vertices of the
	same color. The \textit{chromatic number} of $\mathcal{H}$, denoted by $\chi(\mathcal{H})$, is the minimum
	number of colors needed for a proper vertex-coloring of $\mathcal{H}$. Two simple hypergraphs $\mathcal{H}_1$ and $\mathcal{H}_2$ are called \textit{isomorphic} if there exists a bijection between their vertex sets such that any subset of vertices form a hyperedge in $\mathcal{H}_1$ if and only if the corresponding subset of vertices forms a hyperedge in $\mathcal{H}_2$. If $\mathcal{H}_1$ and $\mathcal{H}_2$ are isomorphic, then we denote it by $\mathcal{H}_1\cong\mathcal{H}_2$. The \textit{empty hypergraph} is the hypergraph with empty vertex set and empty hyperedge set.

	An \textit{embedding of a graph} on a surface is a continuous and  one to one
	function from a topological representation of the graph into the surface. We denote by $S_n$ the surface obtained from the sphere $S_0$  by adding $n$ handles. The number $n$ is called the \textit{genus of the surface} $S_n, n\geq 0$.  The \textit{orientable genus} of a graph $G$, denoted by $g(G)$, is the minimum genus of
	a surface in which $G$ can be embedded.  A \textit{cross-cap} is a topological object formed by identifying opposite points on the boundary of a circle (or a disk) and is equivalent to gluing a Möbius strip into a hole in a surface. A surface obtained by adding $k$ crosscaps to $S_0$ 
	is known as the non-orientable surface and we denote it by $N_k$. The number $k$ is called the crosscap of $N_k$. The non-orientable genus of
	a graph $G$, denoted by $\tilde{g}(G)$, is the smallest integer $k$ such that $G$ can be embedded on $N_k$.  A graph is said to be \textit{planar} if it can be drawn on the plane in such a way that no edges intersect, except at a common end vertex.  A graph is said to be \textit{toroidal} if it  can be embedded on a torus and is called \textit{projective} if it can be embedded on a projective plane.   Further, note that if $H$ is a subgraph of a graph $G$, then $g(H) \leq g(G)$ and $\tilde{g}(H) \leq \tilde{g}(G)$. A hypergraph is \textit{toroidal} if its incidence graph is toroidal and is  \textit{projective} if its incidence graph is projective. For more details on graphs and hypergraphs, one may refer \cite{voloshinhypergraphs,white1985graphs}, etc.
	
	The Cauchy's theorem for abelian groups states that, if $G$ is a group and if a prime $p$ divides $\mid G \mid$, then $G$ has subgroup of order $p$.

In \cite{sb},we have introduced the intersection hypergraph of a group as follows:
		\begin{defn}
		Let $G$ be a group and $S$ be the set of all non-trivial proper subgroups of $G$. \textit{The intersection hypergraph of $G$}, denoted by $\tilde{\Gamma}_\mathcal{H}(G)$, is a hypergraph whose vertex set,\\ $V =\{H \in S \,\, | \,\, H \cap K = \{e\} \,\, \text{for some $ K \in S $} \}$ and $E \subseteq V$ is a hyperedge if and only if \begin{enumerate}
			\item For distinct $H,K \, \in E$, $H \cap K = \{e\}.$
			\item There does not exist  $E' \supset E$ which  satisfies (1).
		\end{enumerate}
		
	\end{defn}
\noindent Following is the definition of co-maximal hypergraph of a group in general.
		\begin{defn}
		Let $G$ be a group and $S$ be the set of all non-trivial proper subgroups of $G$. The \textit{co-maximal hypergraph of $G$}, denoted by $Co_{\mathcal{H}}(G)$, is an undirected hypergraph whose vertex set, $V=\{H \in S \, | \, HK=G \,\, \text{for some} \, K \in S\}$ and $E \subseteq V$ is a hyperedge if and only if \begin{enumerate}
			\item for distinct $H,K \, \in E$, $HK = G.$
			\item there does not exist  $E' \supset E$ which  satisfies (1).
		\end{enumerate}
	\end{defn}
	
	In this paper, our aim is to study the intersection and co-maximal hypergraph on the subgroups of $\mathbb{Z}_n$.  In Section 2, we prove that the intersection and co-maximal hypergraph of a finite abelian group  are isomorphic. In Section 3, we study various structural properties of intersection hypergraph of $\mathbb{Z}_n$. In Section 4,  we  discuss the possibilities of $\tilde{\Gamma}_\mathcal{H}(\mathbb{Z}_n)$ which can be embedded on the plane, torus and projective plane.
	
	\section{Isomorphism between Intersection and Co-maximal Hypergraphs}
	In this section, we establish that the intersection hypergraph and co-maximal hypergraph of a finite abelian group are isomorphic. For this, we need the following definitions and results.\\
		A subgroup $G$ has a \textit{dual} if there exists a duality from its subgroup lattice onto the subgroup lattice of a group $\bar{G}$, i.e., there exists a bijective map $\delta: L(G) \rightarrow L(\bar{G})$ such that for all $H,K \in L(G), H \leq K$ if and only if $\delta(K) \leq \delta(H)$. Let $\sigma$ be an isomorphism from $G$ to $\bar{G}$. Then, the map $\bar{\sigma}:L(G) \rightarrow L(\bar{G})$ defined as $\bar{\sigma}(H)=\sigma(H)$ for all subsets $H$ of $G$ is a projectivity from $G$ to $\bar{G}$ and we call $\bar{\sigma}$ the \textit{projectivity induced by $\sigma$}.
	\begin{lemma}\cite{schmidt}
		Let $G$ be a finite abelian group, written multiplicatively, and let $G^{\star}=Hom(G, \mathbb{C}^\star)$ be the set of all homomorphisms from $G$ into the multiplicative group $\mathbb{C}^{\star}$ 
		of complex numbers. For $\sigma, \tau \in G^\star$, define $\sigma \circ\tau \in G^\star$ by $x^{\sigma \circ\tau} = x^{\sigma}  x^{\tau}$.
		\begin{enumerate}
			\item Then $(G^\star, \circ)$ is a group isomorphic to $G$; it is called the character group of $G$.
			\item  If $1 \neq x \in G$, there exists $\sigma \in G^\star$ such that $x^{\sigma} \neq 1$.
		\end{enumerate}  
	\end{lemma}
 
	\begin{theorem}\label{selfduality}\cite{schmidt} 
		A finite abelian group G is self-dual, that is, there exists 
		a duality from G onto G.
	\end{theorem}
		\noindent In the proof of Theorem \ref{selfduality}, for a subset $X$ of a finite abelian group $G$, $X^\perp$ is defined as,  $$X^{\perp}=\{\sigma \in G^\star \mid x^{\sigma}=1\,\,  \text{for all}\, x \in X\}.$$ It is also denoted by $\perp(X)$. Observe that $\perp$ is a duality from $G$ onto $G^\star$. Let $\rho$ be the projectivity induced by an isomorphism from $G^\star$ onto $G$. Hence, $\delta=\rho\perp$, the composition of $\rho$ and $\perp$, is a duality from $G$ onto $G$.
	\begin{theorem}\label{iso}
		For a finite abelian group $G, \tilde{\Gamma}_{\mathcal{H}}(G) \cong Co_\mathcal{H}(G)$. 
	\end{theorem}
	\begin{proof}
		Let $G$ be a finite abelian group and $H_1 \in V(\tilde{\Gamma}_{\mathcal{H}}(G))$.
		Then, by the definition, there exists $H_2  \in V(\tilde{\Gamma}_{\mathcal{H}}(G))$ such that $H_1 \cap H_2 =\{e\}$. Note that for subgroups $H_1,H_2$ of $G$, $H_1 \vee H_2=$ $<H_1 \cup H_2>$ and $H_1 \wedge H_2 = H_1 \cap H_2$ in $L(G)$, the lattice of subgroups of $G$. As $G$ is abelian, $\delta(H_1)$ and $\delta(H_2)$ are normal. Therefore, 
		$\delta(H_1) \delta(H_2) = \delta(H_1) \vee \delta(H_2) =
			  \rho \perp(H_1) \vee \rho\perp(H_2) = \\ \rho(\perp(H_1) \vee \perp(H_2))   = \rho(\perp(H_1 \wedge H_2)) 
			  = \rho((H_1 \wedge H_2)^{\perp})
			  = \rho(\{e\}^{\perp})
			 = \rho (G^\star)
			 = G$
		Thus, if $H_1 \in V(\tilde{\Gamma}_{\mathcal{H}}(G))$, then $\delta(H_1) \in V(Co_\mathcal{H}(G))$.\\
		Define a map $f:    V(\tilde{\Gamma}_{\mathcal{H}}(G)) \longrightarrow V(Co_\mathcal{H}(G))$ such that $f(H) = \delta(H)$ for each $H \in  V(\tilde{\Gamma}_{\mathcal{H}}(G))$. Since $\delta$ is a well-defined bijective map, $f$ is also a well-defined bijective map.\\ Observe that,  
		$
		H_1 \sim H_2 \,\, \text{in} \, \, \tilde{\Gamma}_{\mathcal{H}(G)} \, \, 
			 \text{if and only if} \, \,  H_1 \cap H_2 =\{e\}
			 \text{if and only if} \, \,(H_1 \cap H_2)^{\perp} = \{e\}^{\perp}
		\, \, 
		 \text{if and}\\ \text{only if} \, \, (H_1 \wedge H_2)^{\perp} = G^\star 
		\, \, 
		 \text{if and only if} \, \, \rho((H_1 \wedge H_2)^{\perp}) = \rho(G^\star)
		\, \, 
		 \text{if and only if} \, \, \rho(\perp(H_1 \wedge H_2)) = G 
			\, \, 
			 \text{if and only if} \, \,\rho(\perp(H_1) \vee \perp(H_2)) = G
			\, \, 
			 \text{if and only if} \, \, \rho \perp(H_1) \vee \rho\perp(H_2) = G 
		\, \, 
		 \text{if and only if}\\ \, \, \delta(H_1) \vee \delta(H_2) = G 
		\, \, 
		 \text{if and only if} \, \,\delta(H_1)  \delta(H_2) = G 
			\, \, 
			 \text{if and only if} \, \, f(H_1)  f(H_2) = G 
		\, \, 
		 \text{if and only if}\\ \, \, f(H_1) \sim f(H_2) \,\, \text{in} \,\, Co_\mathcal{H}(G)$, i.e., $
		 H_1 \sim H_2 \,\, \text{in} \, \, \tilde{\Gamma}_{\mathcal{H}(G)}$ if and only if $f(H_1) \sim f(H_2) \,\, \text{in} \,\, Co_\mathcal{H}(G)$. 
		Hence,  for a positive integer $r$,  $\{H_1, H_2, \ldots, H_r\}$ is a hyperedge  of $\tilde{\Gamma}_{\mathcal{H}}(G)$ if and only if \\ $\{f(H_1),f(H_2), \dots, f(H_r)\}$ is a hyperedge of $Co_\mathcal{H}(G)$.\\
		Thus, $f$ is an isomorphism from $V(\tilde{\Gamma}_{\mathcal{H}}(G))$ to $V(Co_\mathcal{H}(G))$. Thus, $\tilde{\Gamma}_{\mathcal{H}}(G) \cong Co_\mathcal{H}(G)$. 
	\end{proof}
	
	\begin{remark}
		For non-abelian groups, Theorem \ref{iso} need not be true. For example, consider the dihedral group of order 8, $D_4= <a,b\,|a^4=e=b^2, bab^{-1}=a^{-1}>$. The vertex set of $\tilde{\Gamma}_\mathcal{H}(D_4)$ is $ V=\{H_2,H_3,H_4,H_5, H_6,H_7,H_8,H_9\}$, where 
		$H_2=<a^2>, H_3=<b>, H_4=<ab>,H_5= \\ <a^2b>, H_6=<a^3b>, H_7=<a>, H_8=<a^2,b>$ and $ H_9=<a^2,ab> $.  Note that $\tilde{\Gamma}_{\mathcal{H}}(D_4)$ and $Co_\mathcal{H}(D_4)$ are not isomorphic because $\tilde{\Gamma}_{\mathcal{H}}(D_4)$ contains 4 hyperedges whereas $Co_\mathcal{H}(D_4)$ contains 5 hyperedges, see Figure \ref{D4}.
		\begin{figure}[h]
			\centering
			\subfloat[$\tilde{\Gamma}_\mathcal{H}(D_4)$]{\begin{tikzpicture}[scale=.6]
					
					\draw[thick] (0,0) ellipse (3 and .8);
					\draw[thick] (1.5,0) ellipse (3 and .8);
					\draw[thick] (-.5,-.5) ellipse (.8 and 2);
					\draw[thick] (1.5,-.5) ellipse (.8 and 2);

					\node at (-2, 1) {$e_1$};
					\node at (4, 1) {$e_2$};
					\node at (-.5, 1.7) {$e_3$};
					\node at (1.5, 1.7) {$e_4$};

					
					\filldraw (-2,-0.2cm) circle (2pt) node[above] {\tiny{${H_2}$}};
					\filldraw (-1,-0.2cm) circle (2pt) node[above] {\tiny{${H_3}$}};
					\filldraw (0,-0.2cm) circle (2pt) node[above] {\tiny{${H_5}$}};
					\filldraw (4, -0.2cm) circle (2pt) node[above] {\tiny{${H_7}$}};
					\filldraw (1, -0.2cm) circle (2pt) node[above] {\tiny{${H_4}$}};
					\filldraw (2, -0.2cm) circle (2pt) node[above] {\tiny{${H_6}$}};
					\filldraw (-0.5, -1cm) circle (2pt) node[below] {\tiny{${H_9}$}};
					\filldraw (1.5, -1cm) circle (2pt) node[below] {\tiny{${H_8}$}};

			\end{tikzpicture}}
			\hspace{.5cm}
			\subfloat[$Co_\mathcal{H}(D_4)$]{\begin{tikzpicture}[scale=.7]
					
					\draw[rotate=30,thick] (.8,-0.5) ellipse (.7 and 2);
					\draw[thick] (1,0) ellipse (3 and .5);
					\draw[thick] (1,-.8) circle (1.5cm);
					\draw[thick] (2,.2) circle (1.2cm);
					\draw[thick] (2,-.5) circle (1.2cm);

					\node at (-2.2, 0.2) {$e_1$};
					\node at (-.5, 1.5) {$e_2$};
					\node at (2.5, 1.5) {$e_3$};
					\node at (-.6, -1.5) {$e_4$};
					\node at (3, -1.5) {$e_5$};


					\filldraw (-1,-0.2cm) circle (2pt) node[right] {\tiny{${H_3}$}};
					\filldraw (2,.9cm) circle (2pt) node[above] {\tiny{${H_5}$}};
					\filldraw (1.1, -0.2cm) circle (2pt) node[right] {\tiny{${H_7}$}};
					\filldraw (0.2, 1cm) circle (2pt) node[above] {\tiny{${H_4}$}};
					\filldraw (0, -1.2cm) circle (2pt) node[below] {\tiny{${H_6}$}};
					\filldraw (2, -1.2cm) circle (2pt) node[left] {\tiny{${H_9}$}};
					\filldraw (3, -0.2cm) circle (2pt) node[left] {\tiny{${H_8}$}};

			\end{tikzpicture}}
			\caption{}
			\label{D4}
		\end{figure}		
	\end{remark}
	\begin{remark}
		The converse of Theorem \ref{iso} need not be true, i.e., even if $\tilde{\Gamma}_{\mathcal{H}}(G) \cong Co_\mathcal{H}(G)$, G need not be abelian. For example, consider the dihedral group $D_3=<a,b \mid a^3=e=b^2, bab^{-1}=a^{-1}>$ of order 6. Both $\tilde{\Gamma}_{\mathcal{H}}(D_3)$ and $Co_\mathcal{H}(D_3)$ consist of a single hyperedge containing $4$ subgroups - $<a>,$ $<b>,<ab>$ and $<a^2b>$. Hence, $\tilde{\Gamma}_{\mathcal{H}}(D_3) \cong Co_\mathcal{H}(D_3)$, but $D_3$ is not abelian.
	\end{remark}
	\section{Structural Properties of $\tilde{\Gamma}_\mathcal{H}(\mathbb{Z}_n)$}
In this section, we discuss the diameter, girth, chromatic number and characterization for star and hypertree structures of $\tilde{\Gamma}_\mathcal{H}(\mathbb{Z}_n)$.

\noindent We will use the following simple result at various places.
	\begin{lemma}\label{lemma2}
	Let $n=p_1^{\alpha_1}p_2^{\alpha_2}\ldots p_k^{\alpha_k}$, where $p_1,p_2, \ldots,p_k$ are distinct primes and $\alpha_1,\alpha_2,\ldots,\alpha_k$ are non-negative integers.	If $H_1=<p_1^{r_1}p_2^{r_2}\ldots p_k^{r_k}>$ and $H_2=<p_1^{s_1}p_2^{s_2}\ldots p_k^{s_k}>$, where $0 \leq r_i,s_i \leq \alpha_i$ for $i=1,2, \ldots, k$, are two subgroups of  $\mathbb{Z}_n$, then   $lcm(p_1^{r_1}p_2^{r_2}\ldots p_k^{r_k},p_1^{s_1}p_2^{s_2}\ldots p_k^{s_k})=n$ if and only if $H_1 \cap H_2 =\{e\}$.
	\end{lemma}

	\begin{theorem}\label{vertexset}
		For $n=p_1^{\alpha_1}p_2^{\alpha_2}\ldots p_k^{\alpha_k}$, where $p_1,p_2, \ldots,p_k$ are distinct primes and $\alpha_1,\alpha_2,\ldots,\alpha_k$ are non-negative integers, a proper subgroup $<p_1^{r_1}p_2^{r_2}\ldots p_k^{r_k}>$ of $\mathbb{Z}_n$, where $0 \leq r_i,s_i \leq \alpha_i$ for $i=1,2, \ldots, k$, is in the vertex set of $\tilde{\Gamma}_\mathcal{H}(\mathbb{Z}_n)$ if and only if  $r_i = \alpha_i $ for some $i= 1,2, \ldots, k$.
	\end{theorem}
	\begin{proof}
		Suppose  $H=<x>=<p_1^{r_1}p_2^{r_2}\ldots p_k^{r_k}> \in V(\tilde{\Gamma}(\mathbb{Z}_n))$. If $r_i < \alpha_i$ for all $i=1,2, \ldots, k$, then we will show that for all proper non-trivial subgroups $K$ of $\mathbb{Z}_n$, $H \cap K \neq \{e\}$ and, which is a contradiction.\\
		Let $K=<y>=<p_1^{s_1}p_2^{s_2}\ldots p_k^{s_k}>$ be a proper non-trivial subgroups of $\mathbb{Z}_n$. Then, consider  following cases:\\
		\textbf{Case 1.} If $y \mid x$, then  $<x> \subseteq <y>$, i.e., $H \subseteq K$. Therefore, $H \cap K = H \neq \{e\}$.\\
			\textbf{Case 2.} If $x \mid y$, then  $<y> \subseteq <x>$, i.e., $H \supseteq K$. Therefore, $H \cap K = K \neq \{e\}$.\\
			\textbf{Case 3.} If $x \nmid y$ and $y \nmid x$, then $lcm(x,y)\neq n$. Thus, by Lemma \ref{lemma2}, $H \cap K  \neq \{e\}$.\\
			Therefore, in all the three cases, we have $H \cap K \neq \{e\}$ for all proper non-trivial subgroups $K$ of $\mathbb{Z}_n$ and, which is a contradiction.
			
			Conversely, suppose that $<p_1^{r_1}p_2^{r_2}\ldots p_k^{r_k}>$ is a proper non-trivial subgroup of $\mathbb{Z}_n$ with atleast one $r_i = \alpha_i$ for $i=1,2, \ldots, k$. Without loss of generality, assume that $\alpha_1=r_1$. Consider the proper non-trivial subgroup $K=<p_2^{\alpha_2}\ldots p_k^{\alpha_k}>$. Note that $lcm(p_1^{r_1}p_2^{r_2}\ldots p_k^{r_k}, p_2^{\alpha_2}\ldots p_k^{\alpha_k})=n$. Therefore, by Lemma \ref{lemma2}, $ <p_1^{r_1}p_2^{r_2}\ldots p_k^{r_k}>\cap K=\{e\}$ and hence,  $<p_1^{r_1}p_2^{r_2}\ldots p_k^{r_k}> \in  V(\tilde{\Gamma}(\mathbb{Z}_n))$. 
	\end{proof}
	\begin{theorem}
		If $G$ is a finite cyclic group, then $\mid G \mid$ is a power of prime if and only if $\tilde{\Gamma}_\mathcal{H}(G)$ is empty.
	\end{theorem}

	\begin{proof}
		Suppose that $\mid G \mid$ is  of prime power. Then, for any non-trivial proper subgroups $H,K$ of $G$, either $H \subseteq K$ or $K \subseteq H$. Hence, $H \cap K \neq \{e\}$ and therefore, $\tilde{\Gamma}_\mathcal{H}(G)$ is empty.\\
		Conversely, suppose that $\mid G \mid$ is not a power of prime. Then, there exist two distinct prime divisors, say $p$ and $q$, of $|G|$. Hence, by Cauchy’s Theorem for Abelian Groups, there exist two distinct subgroups $H_p$ and  $H_q$ of order $p$ and  $q$ respectively with $H_p \cap H_q=\{e\}$. Thus, $H_p, H_q \in V(\tilde{\Gamma}_\mathcal{H}(G))$ and therefore, $\tilde{\Gamma}_\mathcal{H}(G)$ is non empty.
	\end{proof}
\noindent \textbf{Note:}
		Since $\tilde{\Gamma}_\mathcal{H}(\mathbb{Z}_n)$ is empty for $\omega(n)=1$, we exclude this case and consider $\omega(n) \geq 2$ throughout the article, where $\omega(n)$ denotes the number of distinct prime divisors of $n$.
	
	\begin{theorem}\label{singleedgeset}
		 In $\tilde{\Gamma}_\mathcal{H}(\mathbb{Z}_n)$, $\mid E(\tilde{\Gamma}_\mathcal{H}(\mathbb{Z}_n))\mid = 1$ if and only if $n=p_1p_2$, where $p_1,p_2$ are distinct prime divisors of $n$.
			\end{theorem}
		\begin{proof}
			Suppose that $n=p_1p_2$, where $p_1,p_2$ are distinct prime divisors of $n$. Then, the only vertices of $\tilde{\Gamma}_\mathcal{H}(\mathbb{Z}_n)$ are $<p_1>$ and $<p_2>$ with $<p_1> \cap <p_2>=\{e\}$. Therefore, $E(\tilde{\Gamma}_\mathcal{H}(\mathbb{Z}_n))=\{e_1\}$, where $e_1 = \{<p_1>,<p_2>\}$.
			
			Converely, assume that $n \neq p_1p_2$, where $p_1,p_2$ are distinct prime divisors of $n$. Then, consider the following cases:\\
			\textbf{Case 1.} Suppose that $\omega(n)=2$ and $n=p_1^{\alpha_1}p_2^{\alpha_2}$, where $p_1,p_2$ are distinct prime divisors of $n$ and either $\alpha_1 \geq 2$ or  $\alpha_2 \geq 2$. Without loss of generality, assume that $\alpha_2 \geq 2$. Then, $<p_1^{\alpha_1}p_2^{\alpha_2 -1}>$ and $<p_1^{\alpha_1}p_2^{\alpha_2 -2}>$ are two distinct vertices of $\tilde{\Gamma}_\mathcal{H}(\mathbb{Z}_n)$ with $<p_1^{\alpha_1}p_2^{\alpha_2 -1}> \cap <p_1^{\alpha_1}p_2^{\alpha_2 -2}> =\\ <p_1^{\alpha_1}p_2^{\alpha_1-1}> \neq \{e\}$. Therefore, $<p_1^{\alpha_1}p_2^{\alpha_2 -1}>$ and $<p_1^{\alpha_1}p_2^{\alpha_2 -2}>$ are not adjacent. Hence, there exist two distinct hyperedges $e_1$ and $e_2$ such that $e_1$ containing $<p_1^{\alpha_1}p_2^{\alpha_2 -1}>$ and $e_2$ containing $<p_1^{\alpha_1}p_2^{\alpha_2 -2}>$. Thus,  $\mid E(\tilde{\Gamma}_\mathcal{H}(\mathbb{Z}_n))\mid > 1$.\\
			\textbf{Case 2.}   Suppose that $\omega(n) \geq 3$ and   $n=p_1^{\alpha_1}p_2^{\alpha_2}p_3^{\alpha_3}\ldots p_k^{\alpha_k}$, where $p_1,p_2, \ldots, p_k$ are distinct primes and $k\geq 3$. Observe that $<p_1^{\alpha_1}>$ and $<p_2^{\alpha_2}>$ are two distinct vertices of $\tilde{\Gamma}_\mathcal{H}(\mathbb{Z}_n)$ with $<p_1^{\alpha_1}> \cap  <p_2^{\alpha_2}> = <p_1^{\alpha_1}p_2^{\alpha_2}> \neq \{e\}$. Hence, there exist two distinct hyperedges  such that one contains $<p_1^{\alpha_1}>$ and the other contains $<p_2^{\alpha_2}>$. Thus,  $\mid E(\tilde{\Gamma}_\mathcal{H}(\mathbb{Z}_n))\mid > 1$.
		\end{proof}

	\begin{theorem}
		$\tilde{\Gamma}_\mathcal{H}(\mathbb{Z}_n)$ is connected with $diam(\tilde{\Gamma}_\mathcal{H}(\mathbb{Z}_n)) \leq 3$. In particular, \\
		 $$diam(\tilde{\Gamma}_\mathcal{H}(\mathbb{Z}_n)) = \begin{cases*}
			1, \, \text{if $\omega(n)=2$ and $n=p_1p_2$},\\
			2, \, \text{if $\omega(n)=2$ and $n \neq p_1p_2$},\\
			3, \,\,  \text{if} \, \omega(n)\geq 3.
		\end{cases*}$$
	\end{theorem}
	\begin{proof}
		Consider the following cases:\\
		\textbf{Case 1.} Suppose that $\omega(n)=2$ and $n=p_1p_2$. By Theorem \ref{singleedgeset}, $\mid E(\tilde{\Gamma}_\mathcal{H}(\mathbb{Z}_n))\mid = 1$ and, hence $diam(\tilde{\Gamma}_\mathcal{H}(\mathbb{Z}_n))=1$.\\
		\textbf{Case 2.} Suppose that $\omega(n)=2$ and $n=p_1^{\alpha_1} p_2^{\alpha_2}$, where $p_1, p_2$ are distinct primes and either $\alpha_1 \geq 2$ or  $\alpha_2 \geq 2$. Let $H_1=<p_1^{\beta_1} p_2^{\beta_2}>$ and $H_2=<p_1^{\gamma_1} p_2^{\gamma_2}>$, where $0 \leq \beta_1,\gamma_1 \leq \alpha_1$ and $0 \leq \beta_2,\gamma_2 \leq \alpha_2$, be two vertices of 	$\tilde{\Gamma}_\mathcal{H}(\mathbb{Z}_n)$. By Theorem \ref{vertexset}, either $\beta_1=\alpha_1$ or $\beta_2=\alpha_2$ as well as either  $\gamma_1=\alpha_1$ or $\gamma_2=\alpha_2$. It should be noted that the conditions $\beta_1=\alpha_1$ and $\beta_2=\alpha_2$ cannot hold simultaneously, nor $\gamma_1=\alpha_1$ and $\gamma_2=\alpha_2$. \\
		\textbf{Subcase 2.1.} Assume  $\beta_1=\alpha_1$ and  $\gamma_1=\alpha_1$. Then, $H_1 \cap H_2 = <p_1^{\alpha_1}p_2^{lcm(\beta_2,\gamma_2)}> \neq \{e\}$.  Consider the vertex $K=<p_2^{\alpha_2}>$. Hence, by Lemma \ref{lemma2}, $H_1 \cap K =\{e\}$ and $H_2 \cap K =\{e\}$. Therefore, there exist two hyperedges $e_1$ and $e_2$ such that $e_1$ containing $H_1$ and $K$, and $e_2$ containing $H_2$ and $K$. Thus, $H_1 e_1 K e_2 H_2$ is a path from $H_1$ to $H_2$ and hence, $dist(H_1,H_2)=2$. \\
			\textbf{Subcase 2.2.} Assume  $\beta_1=\alpha_1$ and  $\gamma_2=\alpha_2$. Then, by Lemma \ref{lemma2}, $H_1 \cap H_2 = \{e\}$. Hence, there exists a hyperedge $e_1$ containing $H_1$ and $H_2$. Therefore, $H_1 e_1 H_2$ is a path from $H_1$ to $H_2$ and thus,  $dist(H_1,H_2)=1$.\\
			\textbf{Subcase 2.3.} Assume  $\beta_2=\alpha_2$ and  $\gamma_1=\alpha_1$. Then, by Lemma \ref{lemma2}, $H_1 \cap H_2 = \{e\}$. Hence, there exists a hyperedge $e_1$ containing $H_1$ and $H_2$. Therefore, $H_1 e_1 H_2$ is a path from $H_1$ to $H_2$ and thus,  $dist(H_1,H_2)=1$.\\
				\textbf{Subcase 2.4.} Assume  $\beta_2=\alpha_2$ and  $\gamma_2=\alpha_2$. Then, by Lemma \ref{lemma2}, $H_1 \cap H_2 = <p_1^{lcm(\alpha_1,\gamma_1)}p_2^{\alpha_2}> \neq \{e\}$.  Consider the vertex $K=<p_1^{\alpha_1}>$. Hence, $H_1 \cap K =\{e\}$ and $H_2 \cap K =\{e\}$. Therefore, there exist two hyperedges $e_1$ and $e_2$ such that $e_1$ containing $H_1$ and $K$, and $e_2$ containing $H_2$ and $K$. Thus, $H_1 e_1 K e_2 H_2$ is a path from $H_1$ to $H_2$ and hence, $dist(H_1,H_2)=2$.

				\noindent\textbf{Case 3.} Suppose that  $\omega(n) \geq 3$ and  $n=p_1^{\alpha_1} p_2^{\alpha_2} \ldots p_k^{\alpha_k},$ where $p_1,p_2, \ldots, p_k$ are distinct primes, $\alpha_1, \alpha_2, \ldots, \alpha_k$ are positive integers and $k\geq 3$. Let $H_1=<p_1^{\beta_1} p_2^{\beta_2} \ldots p^{\beta_k}>$ and  $H_2=$ \\ $<p_1^{\gamma_1} p_2^{\gamma_2} \ldots p^{\gamma_k}>$, where $0 \leq \beta_i,\gamma_i \leq \alpha_i$ for $i=1,2, \ldots, k$, be two vertices of 	$\tilde{\Gamma}_\mathcal{H}(\mathbb{Z}_n)$. If $H_1 \cap H_2 =\{e\}$, then there exists a hyperedge containing $H_1$ and $H_2$. Therefore, $dist(H_1,H_2)=1.$ If $H_1 \cap H_2 \neq \{e\}$, then consider the following subcases: \\ 
				\textbf{Subcase 3.1.} Suppose $\beta_i=\alpha_i=\gamma_i$ for some $i$. Without loss of generality, assume that $\beta_1=\alpha_1=\gamma_1$. Consider the vertex $K=<p_2^{\alpha_2} \ldots p_k^{\alpha_k}>$. Observe that, by Lemma \ref{lemma2}, $H_1 \cap K =\{e\}$ and $H_2 \cap K =\{e\}$. Hence, there exist two hyperedges $e_1$ and $e_2$ such that $e_1$ containing $H_1$ and $K$, and $e_2$ containing $H_2$ and $K$. Thus, $H_1 e_1 K e_2 H_2$ is a path from $H_1$ to $H_2$ and therefore, $dist(H_1,H_2)=2$.\\
				\textbf{Subcase 3.2.} Suppose that whenever  $\beta_i = \alpha_i$, $\alpha_i \neq \gamma_i$. Assume that $\beta_m =\alpha_m$ and $\gamma_r =\alpha_r$, where $1 \leq m,r \leq k$. Consider the vertex  $K_1=<p_1^{\alpha_1}p_2^{\alpha_2}\ldots p_{m-1}^{\alpha_{m-1}} p_{m+1}^{\alpha_{m+1}} \ldots p_k^{\alpha_k}>$ of $\tilde{\Gamma}_\mathcal{H}(\mathbb{Z}_n)$. Clearly, $H_1 \cap K_1 =\{e\}$. Now, consider the following subsubcases:\\
				\textbf{Subsubcase 3.2.1.} If $H_2 \cap K_1=\{e\}$, then  there exist two hyperedges $e_1$ and $e_2$ such that $e_1$ containing $H_1$ and $K_1$, and $e_2$ containing $H_2$ and $K_1$. Thus, $H_1 e_1 K_1 e_2 H_2$ is a path from $H_1$ to $H_2$ and therefore, $dist(H_1,H_2)=2$.\\
				\textbf{Subsubcase 3.2.2.} If  $H_2 \cap K_1 \neq \{e\}$, then consider the vertex $K_2=<p_1^{\alpha_1}p_2^{\alpha_2}\ldots p_{r-1}^{\alpha_{r-1}} p_{r+1}^{\alpha_{r+1}} \ldots p_k^{\alpha_k}>$.  Observe that $K_1 \cap K_2 =\{e\}$ and $H_2 \cap K_2 =\{e\}$. Hence, there exist three hyperedges $e_1,e_2$ and $e_3$ such that $e_1$ containing $H_1$ and $K_1$, $e_2$ containing $K_1$ and $K_2$, and $e_3$ containing $H_2$ and $K_2$. Thus, $H_1 e_1 K_1 e_2 K_2 e_3 H_2$ is a path from $H_1$ to $H_2$ and thus, $dist(H_1,H_2)=3$.\\
				Thus, from all the cases above, we have $\tilde{\Gamma}_\mathcal{H}(\mathbb{Z}_n)$ is connected with $diam(\tilde{\Gamma}_\mathcal{H}(\mathbb{Z}_n)) \leq 3$.
			\end{proof}
			
	\begin{theorem}\label{starhyp}
		$\tilde{\Gamma}_\mathcal{H}(\mathbb{Z}_n)$ is a star hypergraph  if and only if $n=p_1^{\alpha_1} p_2$, where $p_1,p_2$ are distinct prime numbers and $\alpha_1$ is a positive integer.
	\end{theorem}
	\begin{proof}
	Suppose $n=p_1^{\alpha_1} p_2$, where $p_1,p_2$ are distinct prime numbers and $\alpha_1$ is a positive integer. Consider the vertex $<p_1^{\alpha_1}>$ of $\tilde{\Gamma}_\mathcal{H}(\mathbb{Z}_n)$. By Theorem \ref{vertexset}, a vertex of $\tilde{\Gamma}_\mathcal{H}(\mathbb{Z}_n)$ different from $<p_1^{\alpha_1}>$ is in the form of $<p_1^l p_2>$, where $0 \leq l \leq \alpha_1-1$.   Then, as $lcm(p_1^{l}p_2,p_1^{\alpha_1})=p_1^{\alpha_1} p_2$, by Lemma \ref{lemma2}, we have $<p_1^l p_2> \, \cap $ $<p_1^{\alpha_1}> =\{e\}$. 
Therefore,  $<p_1^{\alpha_1}>$ must belongs to all the hyperedges of $\tilde{\Gamma}_\mathcal{H}(\mathbb{Z}_n)$ by the maximality condition of the hyperedges. Thus, $\tilde{\Gamma}_\mathcal{H}(\mathbb{Z}_n)$ is a star hypergraph.

	\noindent Conversely, suppose that  $n \neq p_1^{\alpha_1} p_2$, where $p_1,p_2$ are distinct prime numbers and $\alpha_1$ is a positive integer. Then, consider the following cases:\\
	\textbf{Case 1.} Suppose $\omega(n)= 2$ and  $n=p_1^{\alpha_1} p_2^{\alpha_2}$, where $p_1,p_2$ are distinct primes and $\alpha_1, \alpha_2 \geq 2$. Note that $V(\tilde{\Gamma}_\mathcal{H}(\mathbb{Z}_n))=A \cup B$, where 
	$A$ is the set of  vertices of the form $<p_1^{\alpha_1}>, <p_1^{\alpha_1}p_2>, \ldots,$ $<p_1^{\alpha_1}p_2^{\alpha_2-1}>$ and  B is the set of vertices of the form $<p_2^{\alpha_2}>, <p_1 p_2^{\alpha_2}>,\ldots, <p_1^{\alpha_1-1}p_2^{\alpha_2}>$. Observe that no two vertices of the set A  and no two vertices of the set B are adjacent. Hence, no vertex of $\tilde{\Gamma}_\mathcal{H}(\mathbb{Z}_n)$ belongs to all the hyperedges of $\tilde{\Gamma}_\mathcal{H}(\mathbb{Z}_n)$. Therefore, $\tilde{\Gamma}_\mathcal{H}(\mathbb{Z}_n)$ is not a star hypergraph.\\
		\textbf{Case 2.} Suppose $\omega(n) \geq 3$ and   $n=p_1^{\alpha_1}p_2^{\alpha_2}p_3^{\alpha_3}\ldots p_k^{\alpha_k}$, where $p_1,p_2, \ldots, p_k$ are distinct primes and $k\geq 3$. Let $H=<p_1^{\beta_1}p_2^{\beta_2}p_3^{\beta_3}\ldots p_k^{\beta_k}>$, where $0 \leq \beta_i \leq \alpha_i$ for all $i=1,2,3, \ldots, k$ be an arbitrary vertex of $\tilde{\Gamma}_\mathcal{H}(\mathbb{Z}_n)$. Since $H$ is a non-trivial subgroup of $\mathbb{Z}_n$,   $\beta_i \neq \alpha_i$ for some $i \in \{1,2, \ldots k\}$.	Without loss of generality, assume that $\beta_1 \neq \alpha_1$. Then, $ H \, \cap <p_1^{\beta_1}p_2^{\alpha_2}> = <p_1^{\beta_1}p_2^{\alpha_2}p_3^{\beta_3}\ldots p_k^{\beta_k}>\neq \{e\}$ and $ H \, \cap <p_1^{\beta_1}p_2^{\alpha_2}p_3^{\alpha_3}> = <p_1^{\beta_1}p_2^{\alpha_2}p_3^{\alpha_3}\ldots p_k^{\beta_k}> \neq \{e\}$. Thus, $H$ does not belong to all the hyperedges of $\tilde{\Gamma}_\mathcal{H}(\mathbb{Z}_n)$. Since $H$ is an arbitrary vertex of $\tilde{\Gamma}_\mathcal{H}(\mathbb{Z}_n)$, no vertex of $\tilde{\Gamma}_\mathcal{H}(\mathbb{Z}_n)$ belong to all the hyperedges of $\tilde{\Gamma}_\mathcal{H}(\mathbb{Z}_n)$. Therefore, $\tilde{\Gamma}_\mathcal{H}(\mathbb{Z}_n)$ is not a star hypergraph.
		\end{proof}
	
		\noindent Note that every star hypergraph is a hypertree. But the converse need not be true in general. The following theorem characterizes when 	$\tilde{\Gamma}_\mathcal{H}(\mathbb{Z}_n)$ is a hypertree. To characterize 	$\tilde{\Gamma}_\mathcal{H}(\mathbb{Z}_n)$ as a hypertree, we need the following definition.
		\begin{defn}\cite{voloshinhypergraphs}
			A \textit{host graph} $G$ for a hypergraph is a connected graph on the same vertex set such that
			every hyperedge induces a connected subgraph of $G$. A hypergraph $\mathcal{H} = (X,\mathcal{D})$ is called a \textit{hypertree} if there exists a host tree
			$T = (X,E)$ such that each edge $D \in \mathcal{D}$  induces a subtree in $T$. 
		\end{defn}
		
	\begin{theorem}
		$\tilde{\Gamma}_\mathcal{H}(\mathbb{Z}_n)$ is a hypertree iff $n=p_1^{\alpha_1} p_2$ or $n=p_1p_2p_3$, where $p_1,p_2,p_3$ are distinct primes and $\alpha_1 \geq 1$.
	\end{theorem}
	\begin{proof} Consider the following cases:\\
		\textbf{Case 1.} Suppose that $n=p_1^{\alpha_1} p_2$ , where $p_1,p_2$ are distinct primes and $\alpha_1 \geq 1$. Then, by Theorem \ref{starhyp}, $\tilde{\Gamma}_\mathcal{H}(\mathbb{Z}_n)$ is a star hypergraph and therefore, $\tilde{\Gamma}_\mathcal{H}(\mathbb{Z}_n)$ is a hypertree.

		\noindent\textbf{Case 2.} Now, suppose that $n=p_1p_2p_3$,  where $p_1,p_2,p_3$ are distinct primes. The Figure \ref{hosttree}(a) depicts a host tree for $\tilde{\Gamma}_\mathcal{H}(\mathbb{Z}_n)$. Hence, $\tilde{\Gamma}_\mathcal{H}(\mathbb{Z}_n)$ is a hypertree.

		\noindent Conversely, suppose that  $n=p_1^{\alpha_1}p_2^{\alpha_2}p_3^{\alpha_3}\ldots p_k^{\alpha_k}$, where $p_1,p_2, \ldots, p_k$ are distinct primes, $\alpha_1, \alpha_2,$ $ \ldots, \alpha_k$ are non-negative integers and $k\geq 2$ with $n \neq p_1^{\alpha_1}p_2$ and  $n \neq p_1 p_2 p_3$. We will prove that there is no host tree for $\tilde{\Gamma}_\mathcal{H}(\mathbb{Z}_n)$. Let $\mathcal{G}_\mathcal{H}$ be a host graph for $\tilde{\Gamma}_\mathcal{H}(\mathbb{Z}_n)$. Note that any two vertices in the same hyperedge of  $\tilde{\Gamma}_\mathcal{H}(\mathbb{Z}_n)$ have a path  in $\mathcal{G}_\mathcal{H}$.  Consider the vertices $H_1=<p_1^{\alpha_1}>, H_2= \\<p_2^{\alpha_2}p_3^{\alpha_3}\ldots p_k^{\alpha_k}>,$ $ H_3=<p_1^{\alpha_1}p_2^{\alpha_2 -1}p_3^{\alpha_3}\ldots p_k^{\alpha_k}>$ and $H_4=<p_1^{\alpha_1 -1} p_2^{\alpha_2}p_3^{\alpha_3}\ldots p_k^{\alpha_k}>$. Observe that $H_{i} \cap H_{i+1} =\{e\}, \forall i \in \{1,2,3\}$ and $H_1 \cap H_4=\{e\}$. Thus, there exist a path $P_1$ from $H_1$ to $H_2$, a path  $P_2$  from $H_2$ to $H_3$, a path $P_3$ from $H_3$ to $H_4$ and a path $P_4$ from $H_4$ to $H_1$ in $\mathcal{G}_\mathcal{H}$. Therefore, $C=P_1 \cup P_2 \cup P_3 \cup P_4$ is a cycle in $G_\mathcal{H}$  and hence,  $G_\mathcal{H}$ is not a tree, see Figure \ref{hosttree}(b). Since $G_\mathcal{H}$ is an arbitrary host graph of $\tilde{\Gamma}_\mathcal{H}(\mathbb{Z}_n)$, we conclude that there does not exist a host tree for  $\tilde{\Gamma}_\mathcal{H}(\mathbb{Z}_n)$. Thus, $\tilde{\Gamma}_\mathcal{H}(\mathbb{Z}_n)$ is not a hypertree.
			\begin{figure}[h]
			\centering
			
			\subfloat[]{\begin{tikzpicture}[scale=.6]
					
					

					\filldraw (-.7,-0.2cm) circle (2pt) node[above] {\footnotesize{$<p_2p_3>$}};
					
					\filldraw (3.8, -0.2cm) circle (2pt) node[above] {\footnotesize{$<p_1p_2>$}};
					\filldraw (1.5, -0.2cm) circle (2pt) node[above] {\footnotesize{$<p_1p_3>$}};
					\filldraw (3.8, -1.6cm) circle (2pt) node[below] {\footnotesize{$<p_3>$}};
					\filldraw (-0.7, -1.6cm) circle (2pt) node[below] {\footnotesize{$<p_1>$}};
					\filldraw (1.5, -1.6cm) circle (2pt) node[below] {\footnotesize{$<p_2>$}};
					
					\draw (-.7,-0.2cm) -- (1.5, -0.2cm);
					\draw (1.5, -0.2cm) -- (3.8, -0.2cm);
					\draw (-.7,-0.2cm) -- (-0.7, -1.6cm);
					\draw (1.5, -0.2cm) -- (1.5, -1.6cm);
					\draw (3.8, -1.6cm) -- (3.8, -0.2cm);
			\end{tikzpicture}}
			\hspace{.3cm}
			\subfloat[]{\begin{tikzpicture}[scale=.6]
					


					\filldraw (-.5,-0.2cm) circle (2pt) node[left] {\footnotesize{$<p_1^{\alpha_1}>$}};
					
					\filldraw (3.5, -0.2cm) circle (2pt) node[right] {\footnotesize{$<p_2^{\alpha_2}p_3^{\alpha_3}\ldots p_k^{\alpha_k}>$}};
					
					\filldraw (3.5, -1.6cm) circle (2pt) node[right] {\footnotesize{$<p_1^{\alpha_1}p_2^{\alpha_2 -1}p_3^{\alpha_3}\ldots p_k^{\alpha_k}>$}};
					\filldraw (-0.5, -1.6cm) circle (2pt) node[left] {\footnotesize{$<p_1^{\alpha_1 -1} p_2^{\alpha_2}p_3^{\alpha_3}\ldots p_k^{\alpha_k}>$}};
					
					\node at (1.5, 0.1cm) {\footnotesize{$P_1$}};
					\node at (3.9, -.9cm) {\footnotesize{$P_2$}};
					\node at (1.5, -1.9cm) {\footnotesize{$P_3$}};
					\node at (-.9, -.9cm) {\footnotesize{$P_4$}};

					\draw (-.5,-0.2cm) -- (1.5, -0.2cm);
					\draw (1.5, -0.2cm) -- (3.5, -0.2cm);
					\draw (-.5,-0.2cm) -- (-0.5, -1.6cm);
					\draw (3.5, -1.6cm) -- (-0.5, -1.6cm);
					\draw (3.5, -1.6cm) -- (3.5, -0.2cm);
			\end{tikzpicture}}
			\caption{}
			\label{hosttree}
			
		\end{figure}		
		\end{proof}
	
	\begin{corollary}
		If $n=p_1p_2p_3$, where $p_1,p_2$ and $p_3$ are distinct primes, then $\tilde{\Gamma}_\mathcal{H}(\mathbb{Z}_n)$ is a hypertree, but not a star hypergraph.
	\end{corollary}
	\begin{corollary}
		If $\omega(n) > 3$, then $\tilde{\Gamma}_\mathcal{H}(\mathbb{Z}_n)$ is not a hypertree. 
	\end{corollary}
	\noindent Before analyzing the girth $gr(\tilde{\Gamma}_\mathcal{H}(\mathbb{Z}_n))$ of $\mathbb{Z}_n$, let us recall the complement of the intersection graph of a group. In \cite{devicomplement}, the complement of the intersection graph of subgroups of a group G, denoted by  $\mathscr{I}^c(G)$, is the graph whose vertex set is the set of all nontrivial proper subgroups of
	G and its distinct vertices H and K are adjacent if and only if $H \cap K =\{e\}$. Let  $\mathscr{I}^c(G)^{\star}$ denote the graph obtained from  $\mathscr{I}^c(G)$ by removing all the isolated vertices. Note that $gr(\mathscr{I}^c(G)^{\star})=gr(\mathscr{I}^c(G))$ for all $G$.

			\begin{lemma}\cite{devicomplement}\label{trianglefree}
			Let $G$ be a finite group. Then, the following are equivalent.
			\begin{enumerate}
				\item $G$ is isomorphic to one of $\mathbb{Z}_{p^\alpha}(\alpha \geq 1), Q_{2^\alpha}$ or $\mid G \mid = p^\alpha q^\beta(\alpha,\beta \geq 1)$ with $G$ has a unique subgroup of each of distinct prime orders $p,q$;
				\item $\mathscr{I}^c(G)$ is bipartite;
				\item $\mathscr{I}^c(G)$ is $C_3$- free.
			\end{enumerate}
		\end{lemma}
			\begin{remark}
			$\tilde{\Gamma}_\mathcal{H}(G)$ is the clique hypergraph of $\mathscr{I}^c(G)^{\star}$, i.e., the hyperedges of $\tilde{\Gamma}_\mathcal{H}(G)$  are the maximal cliques of $\mathscr{I}^c(G)^{\star}$. 
		\end{remark}
		\begin{remark}
		Let $\mathcal{G}$ be a graph and $Cl_\mathcal{H}(\mathcal{G})$ denote the clique hypergraph of $\mathcal{G}$. Then, 	$gr(\mathcal{G})=gr(Cl_\mathcal{H}(\mathcal{G}))$ iff $\mathcal{G}$ contains no subgraph isomorphic to the graph shown in Figure \ref{girth}.
			\begin{figure}[h]
				\centering
				\begin{tikzpicture}[scale=1.5, every node/.style={circle, fill=black, minimum size=6pt, inner sep=0pt}]
					
					\node[label=left:{}] (u1) at (-1,.5) {};
					\node[label=below:{}] (u2) at (0,0) {};
					\node[label=right:{}] (u3) at (0,1) {};
					\node[label=left:{}] (u4) at (1,.5) {};

					\draw (u1) -- (u2);
					\draw (u1) -- (u3);
					\draw (u2) -- (u3);
					\draw (u2) -- (u4);
					\draw (u3) -- (u4);
				\end{tikzpicture}
				\caption{}
				\label{girth}
			\end{figure}
			
		\end{remark}
		\begin{theorem}\label{girtheq}
			Let $G$ be a finite group. Then,  $gr(\mathscr{I}^c(G))=gr(\tilde{\Gamma}_\mathcal{H}(G))$ if and only if  $G$ is isomorphic to one of $\mathbb{Z}_{p^\alpha}(\alpha \geq 1), Q_{2^\alpha}$ or $\mid G \mid = p^\alpha q^\beta(\alpha,\beta \geq 1)$ with $G$ has a unique subgroup of each of distinct prime orders $p,q$.
		\end{theorem}
	\begin{theorem}
		Let $p_1,p_2$ and $p_3$ be distinct primes. Then, the girth $gr(\tilde{\Gamma}_\mathcal{H}(\mathbb{Z}_n))$ of $\tilde{\Gamma}_\mathcal{H}(\mathbb{Z}_n)$ is given by, \\	$$gr(\tilde{\Gamma}_\mathcal{H}(\mathbb{Z}_n)) =
		\begin{cases*}
			\infty,  \, \text{if} \, n=p_1^{\alpha_1} p_2, \text{where}\,  \alpha_1 \geq 1 \,  \text{or}\, n=p_1p_2p_3,\\
			4, \,  \text{if}\, n = p_1^{\alpha_1}
			 p_2^{\alpha_2},\text{where} \, \alpha_1, \alpha_2 \geq 2,\\ 
			2, \,  \, \text{otherwise}.
		\end{cases*}$$ 
	\end{theorem}
	\begin{proof}
		Consider the following cases:\\
\textbf{Case 1.} If $ n=p_1^{\alpha_1} p_2, \text{where}\,  \alpha_1 \geq 1$, then by Theorem \ref{starhyp}$, \tilde{\Gamma}_\mathcal{H}(\mathbb{Z}_n)$ is a 2-uniform star hypergraph, i.e., simply a star graph. Hence, there does not exist a cycle and therefore, $gr(\tilde{\Gamma}_\mathcal{H}(\mathbb{Z}_n)) = \infty$.\\
\textbf{Case 2.} If $n=p_1p_2p_3$, then as depicted in the Figure \ref{d41},  $\tilde{\Gamma}_\mathcal{H}(\mathbb{Z}_{pqr})$ does not contain a cycle and therefore, $gr(\tilde{\Gamma}_\mathcal{H}(\mathbb{Z}_n)) = \infty$.\\
\begin{figure}[h]
	\centering
	\begin{tikzpicture}[scale=.6]
		
		\draw[thick] (1.3,0) ellipse (5 and 1);
		\draw[thick] (1.5,-.1) ellipse (1.2 and 2);
		\draw[thick] (-1.2,-.1) ellipse (1.2 and 2);
		\draw[thick] (4.1,0) ellipse (1.2 and 2);


		\filldraw (-1.1,-0.2cm) circle (2pt) node[above] {\footnotesize{$<p_2p_3>$}};
		
		\filldraw (4.1, -0.2cm) circle (2pt) node[above] {\footnotesize{$<p_1p_2>$}};
		\filldraw (1.5, -0.2cm) circle (2pt) node[above] {\footnotesize{$<p_1p_3>$}};
		\filldraw (4, -1.7cm) circle (2pt) node[above] {\footnotesize{$<p_3>$}};
		\filldraw (-1.1, -1.7cm) circle (2pt) node[above] {\footnotesize{$<p_1>$}};
		\filldraw (1.5, -1.8cm) circle (2pt) node[above] {\footnotesize{$<p_2>$}};

	\end{tikzpicture}
	\caption{}
	\label{d41}
	
\end{figure}\\
\noindent \textbf{Case 3.} If $n = p_1^{\alpha_1}
p_2^{\alpha_2},\text{where} \, \alpha_1, \alpha_2 \geq 2,$ then, by Lemma \ref{trianglefree} and Theorem \ref{girtheq}, $ gr(\tilde{\Gamma}_\mathcal{H}(\mathbb{Z}_n)) =  gr(\mathscr{I}^c(\mathbb{Z}_n))\geq 4$. Observe that $<p_1 p_2^{\alpha_2}> \cap <p_1^{\alpha_1} p_2>=\{e\}$,  $<p_1^{\alpha_1} p_2> \cap <p_2^{\alpha_2}>=\{e\}$, $<p_2^{\alpha_2}> \cap <p_1^{\alpha_1}>=\{e\}$ and $<p_1^{\alpha_1}> \cap <p_1 p_2^{\alpha_2}>=\{e\}$. Thus, there exist four hyperedges $e_1,e_2,e_3$ and $e_4$ such that $e_1$ contains $<p_1 p_2^{\alpha_2}>$ and $<p_1^{\alpha_1} p_2>$, $e_2$ contains $<p_1^{\alpha_1} p_2>$ and $<p_2^{\alpha_2}>$,  $e_3$ contains $<p_2^{\alpha_2}>$ and $<p_1^{\alpha_1}>$, and $e_4$ contains $<p_1^{\alpha_1}>$ and $<p_1 p_2^{\alpha_2}>$. Hence, $<p_1 p_2^{\alpha_2}> e_1 <p_1^{\alpha_1} p_2> e_2 <p_2^{\alpha_2}> e_3 <p_1^{\alpha_1}> e_4 <p_1 p_2^{\alpha_2}>$ is a cycle in $\tilde{\Gamma}_\mathcal{H}(\mathbb{Z}_n)$ of length 4 and therefore, $gr(\tilde{\Gamma}_\mathcal{H}(\mathbb{Z}_n))=4$.\\
\textbf{Case 4.} Suppose $n= p_1^{\alpha_1}p_2^{\alpha_2}p_3^{\alpha_3}$, where $\alpha_1, \alpha_2$ and $\alpha_3$ are positive integers and atleast one of $\alpha_i$ is greater than 1. Without loss of generality, assume that $\alpha_1 > 1$. Then, observe that $<p_1^{\alpha_1}p_2^{\alpha_2}> \cap <p_1^{\alpha_1}p_3^{\alpha_3}>=\{e\}, <p_1^{\alpha_1}p_2^{\alpha_2}> \cap <p_2^{\alpha_2}p_3^{\alpha_3}>=\{e\},<p_1^{\alpha_1}p_3^{\alpha_3}> \cap <p_2^{\alpha_2}p_3^{\alpha_3}>=\{e\},  <p_1^{\alpha_1}p_2^{\alpha_2}> \cap <p_1p_2^{\alpha_2}p_3^{\alpha_3}>=\{e\},$ and $<p_1^{\alpha_1}p_3^{\alpha_3}> \cap <p_1p_2^{\alpha_2}p_3^{\alpha_3}>=\{e\}$. Moreover, $<p_2^{\alpha_2}p_3^{\alpha_3}> \cap <p_1p_2^{\alpha_2}p_3^{\alpha_3}> \neq \{e\}$. Hence, there exist two distinct hyperedges $e_1$ and $e_2$ such that $e_1$ containing $<p_1^{\alpha_1}p_2^{\alpha_2}>, <p_1^{\alpha_1}p_3^{\alpha_3}>$ and $<p_2^{\alpha_2}p_3^{\alpha_3}>$, and $e_2$ containing $<p_1^{\alpha_1}p_2^{\alpha_2}>, <p_1^{\alpha_1}p_3^{\alpha_3}>$ and $<p_1p_2^{\alpha_2}p_3^{\alpha_3}>$. Therefore, $<p_1^{\alpha_1}p_2^{\alpha_2}> e_1 <p_1^{\alpha_1}p_3^{\alpha_3}> e_2 <p_1^{\alpha_1}p_2^{\alpha_2}>$ and so, $gr(\tilde{\Gamma}_\mathcal{H}(\mathbb{Z}_n))=2$.\\ 
\textbf{Case 5.}  If $\omega(n) \geq 4$ and  $n=p_1^{\alpha_1}p_2^{\alpha_2}p_3^{\alpha_3}p_4^{\alpha_4}\ldots p_k^{\alpha_k}$, where  $k\geq 4$, $p_1,p_2, \ldots, p_k$ are distinct primes,  $\alpha_1, \alpha_2, \ldots, \alpha_k$ are non-negative integers,  then observe that $<p_1^{\alpha_1}p_2^{\alpha_2}p_3^{\alpha_3}p_5^{\alpha_5}\ldots p_k^{\alpha_k}> \cap $ $<p_1^{\alpha_1}p_2^{\alpha_2}p_4^{\alpha_4}\ldots p_k^{\alpha_k}>=\{e\}, <p_1^{\alpha_1}p_2^{\alpha_2}p_3^{\alpha_3}p_5^{\alpha_5}\ldots p_k^{\alpha_k}> \cap <p_2^{\alpha_2}p_3^{\alpha_3}p_4^{\alpha_4}\ldots p_k^{\alpha_k}>=\{e\},$ \\
$ <p_1^{\alpha_1}p_2^{\alpha_2}p_4^{\alpha_4}\ldots p_k^{\alpha_k}> \cap <p_2^{\alpha_2}p_3^{\alpha_3}p_4^{\alpha_4}\ldots p_k^{\alpha_k}>=\{e\}, <p_1^{\alpha_1}p_2^{\alpha_2}p_3^{\alpha_3}p_5^{\alpha_5}\ldots p_k^{\alpha_k}> \cap <p_3^{\alpha_3}p_4^{\alpha_4}\ldots p_k^{\alpha_k}>$ \\$ =\{e\}$ and $<p_1^{\alpha_1}p_2^{\alpha_2}p_3^{\alpha_3}p_5^{\alpha_5}\ldots p_k^{\alpha_k}> \cap <p_3^{\alpha_3}p_4^{\alpha_4}\ldots p_k^{\alpha_k}>=\{e\}$. Moreover, $<p_2^{\alpha_2}p_3^{\alpha_3}p_4^{\alpha_4}\ldots p_k^{\alpha_k}> \cap <p_3^{\alpha_3}p_4^{\alpha_4}\ldots p_k^{\alpha_k}> \neq \{e\}$. Hence, there exist two distinct hyperedges $e_1$ and $e_2$ such that $e_1$ containing $<p_1^{\alpha_1}p_2^{\alpha_2}p_3^{\alpha_3}p_5^{\alpha_5}\ldots p_k^{\alpha_k}>, <p_1^{\alpha_1}p_2^{\alpha_2}p_4^{\alpha_4}\ldots p_k^{\alpha_k}>$ and $<p_2^{\alpha_2}p_3^{\alpha_3}p_4^{\alpha_4}\ldots p_k^{\alpha_k}>$, and $e_2$ containing $<p_1^{\alpha_1}p_2^{\alpha_2}p_3^{\alpha_3}p_5^{\alpha_5}\ldots p_k^{\alpha_k}>, <p_1^{\alpha_1}p_2^{\alpha_2}p_4^{\alpha_4}\ldots p_k^{\alpha_k}>$ and  $<p_3^{\alpha_3}p_4^{\alpha_4}\ldots p_k^{\alpha_k}>$. Therefore, $p_1^{\alpha_1}p_2^{\alpha_2}p_3^{\alpha_3}p_5^{\alpha_5}\ldots p_k^{\alpha_k}> e_1 <p_1^{\alpha_1}p_2^{\alpha_2}p_4^{\alpha_4}\ldots p_k^{\alpha_k}> e_2 \\ <p_1^{\alpha_1}p_2^{\alpha_2}p_3^{\alpha_3}p_5^{\alpha_5}\ldots p_k^{\alpha_k}>$ and so, $gr(\tilde{\Gamma}_\mathcal{H}(\mathbb{Z}_n))=2$. 
	\end{proof}
	In the following theorem we establish that the chromatic number $\chi(\tilde{\Gamma}_\mathcal{H}(\mathbb{Z}_n))$ of $\mathbb{Z}_n$ is 2 for all $n$.
	\begin{theorem}
		The chromatic number, $\chi(\tilde{\Gamma}_\mathcal{H}(\mathbb{Z}_n))=2$ for all $n$.
	\end{theorem}
	\begin{proof}
		Let $n=p_1^{\alpha_1}p_2^{\alpha_2} \ldots p_k^{\alpha_k}$, where $p_1, p_2, \ldots, p_k$ are distinct primes and $\alpha_1, \alpha_2, \ldots, \alpha_k$ are non-negative integers. Note that  $V(\tilde{\Gamma}_\mathcal{H}(\mathbb{Z}_n))= A \cup B$, where\\ 
		$A \{<p_1^{\alpha_1}p_2^{\beta_2}p_3^{\beta_3}\ldots p_k^{\beta_k}> \, \, \mid  0 \leq \beta_i \leq \alpha_i  \, \text{for $i=2,3,\ldots,k$}\}$ and \\
		$B=\{<p_1^{\beta_1}p_2^{\beta_2}p_3^{\beta_3}\ldots p_k^{\beta_k}> \, \, \mid  0 \leq \beta_1 \leq \alpha_1 -1 \, \text{and} \, 0 \leq \beta_i \leq \alpha_i \, \text{for $i=2,3,\ldots,k$}\}.$\\
	If $e_1$ is a hyperedge of $\tilde{\Gamma}_\mathcal{H}(\mathbb{Z}_n)$ containing the vertices from $A$ only, then, by Lemma \ref{lemma2},\\ $<p_1^{\alpha_1 -1}p_2^{\alpha_2}p_3^{\alpha_3} \ldots p_k^{\alpha_k}> \cap \,  H =\{e\}$, for all $H \in e_1$.  Thus, $<p_1^{\alpha_1 -1}p_2^{\alpha_2}p_3^{\alpha_3} \ldots p_k^{\alpha_k}>$ must belong to $e_1$, otherwise contradicting the maximality condition of $e_1$. Therefore, $e_1$ cannot contain vertices from $A$ only. Also, note that no two vertices of the set $B$ are adjacent. Hence, any hyperedge of $\tilde{\Gamma}_\mathcal{H}(\mathbb{Z}_n)$ cannot contain only vertices from the set $B$. Thus, each hyperedge has at least one vertex from  A and B both. Assign the color $c_1$ to the vertices in $A$ and the color $c_2$ to the vertices in $B$. This is a proper coloring of $\tilde{\Gamma}_\mathcal{H}(D_n)$ and consequently,  
	$\chi(\tilde{\Gamma}_\mathcal{H}(D_n)) = 2$.
	\end{proof}
	
	\section{Planarity and Non-Planarity of the hypergraph}
		
		To discuss the planarity  of  $\tilde{\Gamma}_\mathcal{H}(\mathbb{Z}_n)$, we essentially need the following results.
			\begin{theorem}\cite{walsh1975hypermaps}
			A graph  is planar if and only if it contains no subdivision of $K_5$ or $K_{3,3}$.
		\end{theorem}
		\begin{theorem}\cite{walsh1975hypermaps} 
			A hypergraph is planar if and only if its incidence graph is planar.
		\end{theorem}
	\begin{theorem}\label{planarity}
		Let $n=p_1^{\alpha_1}p_2^{\alpha_2}p_3^{\alpha_3} \ldots p_k^{\alpha_k}$, where $p_1, p_2,p_3, \ldots, p_k$ are distinct primes and $\alpha_1, \alpha_2, \cdots, \alpha_k$ are non-negative integers. Then,  $\tilde{\Gamma}_\mathcal{H}(\mathbb{Z}_n)$ is planar if and only if $n$ is equal to  $p_1^{\alpha_1}p_2$ or $p_1^{\alpha_1}p_2^2$ or  $p_1p_2p_3$ or $p_1^2p_2p_3$.
	\end{theorem}
	\begin{proof}
		Consider the following cases:\\
		\textbf{Case 1.} If $n=p_1^{\alpha_1}p_2$, then by Theorem \ref{starhyp},  $\tilde{\Gamma}_\mathcal{H}(\mathbb{Z}_n)$ is a 2-uniform hypergraph, i.e., star graph. So, $\tilde{\Gamma}_\mathcal{H}(\mathbb{Z}_n)$ is planar. \\
		\textbf{Case 2.} If $n=p_1^{\alpha_1}p_2^2$, where $\alpha_1 \geq 2$, then Figure \ref{embed} (a) depicts an embedding of  	$\tilde{\Gamma}_\mathcal{H}(\mathbb{Z}_n)$ on the plane, where the unlabelled vertices are hyperedges and hence, $\tilde{\Gamma}_\mathcal{H}(\mathbb{Z}_n)$ is planar.\\
	\noindent \textbf{Case 3.} If $n=p_1 p_2 p_3$, then Figure \ref{embed} (b) depicts an embedding of $\tilde{\Gamma}_\mathcal{H}(\mathbb{Z}_n)$ on the plane and hence, $\tilde{\Gamma}_\mathcal{H}(\mathbb{Z}_n)$ is planar.\\
\noindent\textbf{Case 4.} If $n=p_1^2 p_2 p_3$, then Figure \ref{embed} (c) depicts an embedding of $\tilde{\Gamma}_\mathcal{H}(\mathbb{Z}_n)$ on the plane and hence, $\tilde{\Gamma}_\mathcal{H}(\mathbb{Z}_n)$ is planar.\\
\begin{figure}[h]
	\centering\subfloat[Planar embedding of 	$\mathcal{I}(\tilde{\Gamma}_\mathcal{H}(\mathbb{Z}_{p_1^{\alpha_1}p_2^2})$]{	\begin{tikzpicture}[scale=.7, every node/.style={circle, fill=black, minimum size=4pt, inner sep=0pt}]
			
			\node [label=right:{\tiny{$<p_2^2>$}}](u1) at (-.5,-.8) {};
			\node [label=right:{\tiny{$<p_1p_2^2$>}}](u2) at (-.3,0) {};
			\node [label=right:{\tiny{$<p_1^2p_2^2$>}}] (u3) at (0,1) {};
			\node [scale=.4,label=right:{}] (u4) at (0,1.5) {};
			\node [scale=.4,label=right:{}] (u5) at (0,2) {};
			\node[label=right:{\tiny{$<p_1^{\alpha_1}p_2^2>$}}] (u6) at (0,2.5) {};
			
			\node [label=below:{\tiny{$<p_1^{\alpha_1}p_2>$}}] (a1) at (-1,-1.5) {};
			\node [label=left:{\tiny{$<p_1^{\alpha_1}>$}}] (a2) at (-2,-1.5) {};

			\node [label=right:{}] (v1) at (-.5,-1.2) {};
			\node [label=right:{}] (v2) at (-1.5,-1.2) {};
			\node [label=right:{}] (v3) at (1.2,-1) {};
			\node [label=right:{}] (v4) at (-1.5,-.5) {};
			\node [label=right:{}] (v5) at (-1.5,0) {};
			\node [label=right:{}] (v6) at (1.5,0) {};
			\node [label=right:{}] (v7) at (2,1) {};
			\node [label=right:{}] (v8) at (-1.5,2) {};
			\draw (u1) --(v1);
			\draw (a1) --(v1);
			\draw (u1) --(v2);
			\draw (a2) --(v2);
			\draw (u1) --(v1);
			\draw (a1) --(v1);
			\draw (u2) --(v4);
			\draw (a2) --(v4);
			\draw (u2) --(v3);
			\draw (a1) ..  controls (0,-1.5) .. (v3);
			\draw (u3) --(v5);
			\draw (a2) --(v5);
			\draw (u3) --(v6);
			\draw (a1) ..  controls (1.2,-2) .. (v6);
			\draw (a1) ..  controls (2,-2.4) .. (v7);
			\draw (u6) --(v7);
			\draw (a2) ..  controls (-2,0) .. (v8);
			\draw (u6) --(v8);
	\end{tikzpicture}}
	\label{case3}
	\hspace{.1 cm}
	\subfloat[ Planar embedding of 	$\mathcal{I}(\tilde{\Gamma}_\mathcal{H}(\mathbb{Z}_{p_1p_2p_3})$)]{\begin{tikzpicture}[scale=.6, every node/.style={circle, fill=black, minimum size=4pt, inner sep=0pt}]
			
			\node[label=above:\tiny{$<p_1p_2>$}] (u1) at (-1,4) {};
			\node[label= right:\tiny{$<p_1p_3>$}] (u2) at (2,3) {};
			\node[label=below:\tiny{$<p_2p_3>$}] (u3) at (4,2) {};
			\node[label=left:\tiny{$<p_3>$}] (u4) at (-1,3) {};
			\node[label=left:\tiny{$<p_2>$}] (u5) at (-1,2.5) {};
			\node[label=left:\tiny{$<p_1>$}] (u6) at (-1,1.5) {};
			
			\node[label=above:\tiny{$e_1$}] (v1) at (1,4) {};
			\node[label=above:\tiny{$e_2$}] (v2) at (-3,4) {};
			\node[label=above:\tiny{$e_3$}] (v3) at (1,3) {};
			\node[label=below:\tiny{$e_4$}] (v4) at (1,2.5) {};
			\draw (u1) -- (v1);
			\draw (u2) -- (v1);
			\draw (u3) -- (v1);
			\draw (u1) -- (v2);
			\draw (u4) -- (v2);
			\draw (u2) -- (v3);
			\draw (u5) -- (v3);
			\draw (u3) -- (v4);
			\draw (u6) -- (v4);
			\draw (u2) -- (v4);

			\label{case4}
	\end{tikzpicture}}
	\label{case2embed}
	\hspace{.1 cm}
	\subfloat[Planar embedding of 	$\mathcal{I}(\tilde{\Gamma}_\mathcal{H}(\mathbb{Z}_{p_1^2p_2p_3})$]{
		\begin{tikzpicture}[scale=.9, every node/.style={circle, fill=black, minimum size=4pt, inner sep=0pt}]
			
			\node[label=above right:\tiny{$<p_1^2p_2>$}] (u1) at (-1,4) {};
			\node[label=left:\tiny{$<p_1^2p_3>$}] (u2) at (-1,3.5) {};
			\node[label=left:\tiny{$<p_2p_3>$}] (u3) at (1,5) {};
			\node[label=left:\tiny{$<p_1p_2p_3>$}] (u4) at (-1.5,1) {};
			\node[label=above:\tiny{$<p_3>$}] (u5) at (-1,6) {};
			
			\node[label=above left:\tiny{$<p_1p_3>$}] (u6) at (-2,5) {};
			\node[label=below:\tiny{$<p_2>$}] (u7) at (-1,2) {};
			\node[label=below:\tiny{$<p_1p_2>$}] (u8) at (0,2) {};
			\node[label=above left:\tiny{$<p_1^2>$}] (u9) at (0,3) {};
			
			\node[label=below:\tiny{$e_1$}] (v1) at (1,4) {};
			\node[label=left:\tiny{$e_2$}] (v2) at (-3,5) {};
			\node[label=above:\tiny{$e_3$}] (v3) at (1,3) {};
			\node[label=right:\tiny{$e_4$}] (v4) at (1,1) {};
			\node[label=right:\tiny{$e_5$}] (v5) at (-1,5) {};
			\node[label=below:\tiny{$e_6$}] (v6) at (-1.5,5) {};
			\node[label=left:\tiny{$e_7$}] (v7) at (-1.8,3) {};
			\node[label=right:\tiny{$e_8$}] (v8) at (-.8,3) {};

			
			\draw (u1) -- (v1);
			\draw (u2) -- (v1);
			\draw (u3) -- (v1);
			\draw (u1) -- (v2);
			\draw (u2) -- (v2);
			\draw (u4) -- (v2);
			\draw (u9) -- (v3);
			\draw (u3) to[out=0, in=10](v3);
			\draw (u4) -- (v4);
			\draw (u9) -- (v4);
			\draw (u1) -- (v5);
			\draw (u5) -- (v5);
			\draw (u1) -- (v6);
			\draw (u6) -- (v6);
			\draw (u2) -- (v7);
			\draw (u7) -- (v7);
			\draw (u2) -- (v8);
			\draw (u8) -- (v8);

	\end{tikzpicture}}

	\caption{}
	\label{embed}
	
\end{figure}\\
 Conversely, suppose that $n \neq p_1^{\alpha_1}p_2, n \neq p_1^{\alpha_1}p_2^2, n\neq p_1p_2p_3$ and $ n \neq p_1^2p_2p_3$. Then, consider the following cases:\\
 \textbf{Case 1.} If $\omega(n)=2$ and $n=p_1^{\alpha_1}p_2^{\alpha_2}$, where both $\alpha_1$ and $\alpha_2$ are integers greater than 2, then note that 
		$\tilde{\Gamma}_\mathcal{H}(\mathbb{Z}_n)$ contains a subgraph isomorphic to $\mathcal{I}(\tilde{\Gamma}_\mathcal{H}(\mathbb{Z}_{p_1^3p_2^3}))$ and Figure \ref{k352}(a) depicts that  $\mathcal{I}(\tilde{\Gamma}_\mathcal{H}(\mathbb{Z}_{p_1^3p_2^3}))$ is non-planar. Hence,  in this case,  $\tilde{\Gamma}_\mathcal{H}(\mathbb{Z}_n)$ is non-planar.\\
	\noindent\textbf{Case 2.} If $\omega(n)=3$ with $n \neq p_1 p_2 p_3$ and  $n \neq p_1^2 p_2 p_3$, then $\mathcal{I}(\mathbb{Z}_n)$ contains a subgraph isomorphic to  either Figure \ref{k352}(b) or Figure \ref{k352}(c) where both are subdivisions of $K_{3,3}$. Hence, $\tilde{\Gamma}_\mathcal{H}(\mathbb{Z}_n)$ is non-planar.
			\begin{figure}[h!]
				\centering\subfloat[$\mathcal{I}(\tilde{\Gamma}_\mathcal{H}(\mathbb{Z}_{p_1^3p_2^3}))$ on a plane]{\begin{tikzpicture}[scale=.82, every node/.style={circle, fill=black, minimum size=3pt, inner sep=0pt}]
						
						\node [label=right:{\tiny{$<p_2^3>$}}](u1) at (-.3,-1) {};
						\node [label=right:{\tiny{$<p_1p_2^3$>}}](u2) at (0,0) {};
						\node [label=right:{\tiny{$<p_1^2p_2^3$>}}] (u3) at (0,1) {};
						\node[label=left:{\tiny{$<p_1^3p_2^2>$}}] (a3) at (-3,-1.5) {};
						
						\node [label=below:{\tiny{$<p_1^3p_2>$}}] (a1) at (-2,-1.5) {};
						\node [label=below:{\tiny{$<p_1^3>$}}] (a2) at (-1,-1.5) {};

						\node [label=right:{}] (v1) at (-.4,-1.2) {};
						\node [label=right:{}] (v2) at (1.5,-1) {};
						\node [label=right:{}] (v3) at (2.2,-.8) {};
						\node [label=right:{}] (v4) at (-1,-1) {};
						\node [label=right:{}] (v5) at (-1,-.5) {};
						\node [label=right:{}] (v6) at (-1,0.3) {};
						\node [label=right:{}] (v7) at (-2,-1.6) {};
						\node [label=right:{}] (v8) at (-2.3,0) {};
						\node [label=right:{}] (v9) at (-2.5,0.5) {};
						
						\draw (u1) --(v1);
						\draw (a2) --(v1);
						\draw (u2) --(v2);
						\draw (a2) .. controls (0,-1.5) .. (v2);
						\draw (a2) .. controls (0,-2) .. (v3);
						\draw (u3) .. controls (2,0) .. (v3);
						\draw (u1) --(v4);
						\draw (a1) --(v4);
						\draw (u2) --(v5);
						\draw (a1) --(v5);
						\draw (u3) --(v6);
						\draw (a1) --(v6);
						\draw (u1) --(v7);
						\draw (a3) --(v7);
						\draw (u2) --(v8);
						\draw (a3) --(v8);
						\draw (u3) --(v9);
						\draw (a3) --(v9);
				\end{tikzpicture}}
					\label{G21}
				\hspace{.1cm}
			\centering\subfloat[]{\begin{tikzpicture}[scale=.82, every node/.style={circle, fill=black, minimum size=3pt, inner sep=0pt}]
					
					\node[label=left:\footnotesize{$<p_1^{\alpha_1}p_2^{\alpha_2}>$}] (u1) at (-1,4) {};
					\node[label=left:\footnotesize{$<p_2^{\alpha_2}p_3^{\alpha_3}>$}] (u2) at (-1,3.5) {};
					\node[label=left:\footnotesize{$<p_1^{\alpha_1}p_3^{\alpha_3}>$}] (u3) at (-1,3) {};
					\node[label=left:\footnotesize{$<p_1^{\alpha_1}p_2p_3^{\alpha_3}>$}] (u4) at (-1,2.5) {};
					\node[label=left:\footnotesize{$<p_1p_2^{\alpha_2}p_3^{\alpha_3}>$}] (u5) at (-1,2) {};
					\node[label=left:\footnotesize{$<p_1p_2^{\alpha_2}>$}] (u6) at (-1,1.5) {};
					\node[label=left:\footnotesize{$<p_1^{\alpha_1}>$}] (u7) at (-1,1) {};
					\node[label=left:\footnotesize{$<p_2^{\alpha_2}p_3^{\alpha_3}>$}] (u8) at (-1,0.5) {};
					\node[label=right:$e_1$] (v1) at (1,4) {};
					\node[label=right:$e_2$] (v2) at (1,3.5) {};
					\node[label=right:$e_3$] (v3) at (1,3) {};
					\node[label=right:$e_4$] (v4) at (1,2.5) {};
					\node[label=right:$e_5$] (v5) at (1,2) {};
					\node[label=right:$e_6$] (v6) at (1,1.5) {};
					\node[label=right:$e_7$] (v7) at (1,1) {};
					
					\draw (u1) -- (v1);
					\draw (u2) -- (v1);
					\draw (u4) -- (v1);
					\draw (u1) -- (v2);
					\draw (u3) -- (v2);
					\draw (u5) -- (v2);
					\draw (u1) -- (v3);
					\draw (u2) -- (v3);
					\draw (u3) -- (v3);
					\draw (u4) -- (v4);
					\draw (u6) -- (v4);
					\draw (u3) -- (v5);
					\draw (u6) -- (v5);
					\draw (u5) -- (v6);
					\draw (u7) -- (v6);
					\draw (u7) -- (v7);
					\draw (u8) -- (v7);

			\end{tikzpicture}}
			\label{G21}
			\hspace{.1cm}
			\subfloat[]{
				\begin{tikzpicture}[scale=.82, every node/.style={circle, fill=black, minimum size=3pt, inner sep=0pt}]
					
					\node[label=left:\footnotesize{$<p_1^{\alpha_1}p_2^{\alpha_2}>$}] (u1) at (-1,4) {};
					\node[label=left:\footnotesize{$<p_1^{\alpha_1}p_3^{\alpha_3}>$}] (u2) at (-1,3.5) {};
					\node[label=left:\footnotesize{$<p_1^2p_2^{\alpha_2}p_3^{\alpha_3}>$}] (u3) at (-1,3) {};
					\node[label=left:\footnotesize{$<p_1p_2^{\alpha_2}p_3^{\alpha_3}>$}] (u4) at (-1,2.5) {};
					\node[label=left:\footnotesize{$<p_2^{\alpha_2}p_3^{\alpha_3}>$}] (u5) at (-1,2) {};
					\node[label=left:\footnotesize{$<p_1^{\alpha_1}>$}] (u6) at (-1,1.5) {};

					\node[label=right:$e_1$] (v1) at (1,4) {};
					\node[label=right:$e_2$] (v2) at (1,3.5) {};
					\node[label=right:$e_3$] (v3) at (1,3) {};
					\node[label=right:$e_4$] (v4) at (1,2.5) {};
					\node[label=right:$e_5$] (v5) at (1,2) {};
					\node[label=right:$e_6$] (v6) at (1,1.5) {};

					\draw (u1) -- (v1);
					\draw (u2) -- (v1);
					\draw (u3) -- (v1);
					\draw (u1) -- (v2);
					\draw (u2) -- (v2);
					\draw (u4) -- (v2);
					\draw (u1) -- (v3);
					\draw (u2) -- (v3);
					\draw (u5) -- (v3);
					\draw (u3) -- (v4);
					\draw (u6) -- (v4);
					\draw (u5) -- (v5);
					\draw (u6) -- (v5);
					\draw (u6) -- (v6);
					\draw (u4) -- (v6);

			\end{tikzpicture}}
			
			\caption{}
			\label{k352}
			
		\end{figure}\\
		\noindent\textbf{Case 3.} If $\omega(n) \geq 4$, then consider the following cases:\\
		\textbf{Subcase 3.1.} Suppose that  $n=p_1^{\alpha_1}p_2^{\alpha_2}p_3^{\alpha_3}\ldots p_k^{\alpha_k}$, where  $k\geq 4$, $p_1,p_2, \ldots, p_k$ are distinct primes,  $\alpha_1, \alpha_2, \ldots, \alpha_k$ are non-negative integers and $0 \leq \alpha_i \leq 2$, then it contains a subdivision of $K_{3,3}$ as depicted in Figure \ref{K33}. Hence, $\tilde{\Gamma}_\mathcal{H}(\mathbb{Z}_n)$ is non-planar.
\begin{figure}[h]
	\centering
	\begin{tikzpicture}[scale=1, every node/.style={circle, fill=black, minimum size=3pt, inner sep=0pt}]
		
		\node[label=left:$<p_1^{\alpha_1}p_2^{\alpha_2}p_3^{\alpha_3}p_5^{\alpha_5}\ldots p_k^{\alpha_k}>$] (u1) at (-1,4) {};
		\node[label=left:$<p_1^{\alpha_1}p_3^{\alpha_3}p_4^{\alpha_4}\ldots p_k^{\alpha_k}>$] (u3) at (-1,3) {};
		\node[label=left:$<p_1^{\alpha_1}p_2^{\alpha_2}p_4^{\alpha_4}\ldots p_k^{\alpha_k}>$] (u2) at (-1,3.5) {};
		\node[label=left:$<p_1^{\alpha_1}p_3^{\alpha_3}p_5^{\alpha_5}\ldots p_k^{\alpha_k}>$] (u4) at (-1,2.5) {};
		\node[label=left:$<p_2^{\alpha_2}p_4^{\alpha_4}\ldots p_k^{\alpha_k}>$] (u5) at (-1,2) {};
		
		\node[label=left:$<p_3^{\alpha_3}p_4^{\alpha_4}\ldots p_k^{\alpha_k}>$] (u7) at (-1,1.5) {};
		\node[label=left:$<p_1^{\alpha_1}p_2^{\alpha_2}p_5^{\alpha_5}\ldots p_k^{\alpha_k}>$] (u8) at (-1,1) {};
		\node[label=right:$e_1$] (v1) at (1,4) {};
		\node[label=right:$e_2$] (v2) at (1,3.5) {};
		\node[label=right:$e_3$] (v3) at (1,3) {};
		\node[label=right:$e_4$] (v4) at (1,2.5) {};
		\node[label=right:$e_5$] (v5) at (1,2) {};
		\node[label=right:$e_6$] (v6) at (1,1.5) {};
		\node[label=right:$e_7$] (v7) at (1,1) {};
		
		\draw (u1) -- (v1);
		\draw (u2) -- (v1);
		\draw (u3) -- (v1);
		\draw (u1) -- (v2);
		\draw (u2) -- (v2);
		\draw (u7) -- (v2);
		\draw (u1) -- (v3);
		\draw (u3) -- (v3);
		\draw (u5) -- (v3);
		\draw (u4) -- (v4);
		\draw (u5) -- (v4);
		\draw (u2) -- (v5);
		\draw (u4) -- (v5);
		\draw (u8) -- (v6);
		\draw (u7) -- (v6);
		\draw (u3) -- (v7);
		\draw (u8) -- (v7);

	\end{tikzpicture}

		\caption{}
		\label{K33}

		\end{figure}
		
\noindent \textbf{Subcase 3.2.} Suppose that  $n=p_1^{\alpha_1}p_2^{\alpha_2}p_3^{\alpha_3}\ldots p_k^{\alpha_k}$, where  $k\geq 4$, $p_1,p_2, \ldots, p_k$ are distinct primes,  $\alpha_1, \alpha_2, \ldots, \alpha_k$ are non-negative integers and atleast one of $\alpha_i$ is greater than or equal to 3. Without loss of generality, assume that $\alpha_1\geq 3$.  Consider the vertices  $H_1=$$<p_1^{\alpha_1}p_3^{\alpha_3}p_4^{\alpha_4}\ldots p_k^{\alpha_k}>$, $H_2=<p_1^{\alpha_1}p_2^{\alpha_2}p_4^{\alpha_4}\ldots p_k^{\alpha_k}>$, $H_3=<p_1^{\alpha_1}p_2^{\alpha_2}p_3^{\alpha_3}p_5^{\alpha_5} \ldots p_k^{\alpha_k}>$, $K_1=$ $<p_2^{\alpha_2}p_3^{\alpha_3}\ldots p_k^{\alpha_k}>$,
	$K_2=<p_1p_2^{\alpha_2}p_3^{\alpha_3}\ldots p_k^{\alpha_k}>$ and 
	$K_3=<p_1^2p_2^{\alpha_2}p_3^{\alpha_3}\ldots p_k^{\alpha_k}>$. Observe that   $H_i \cap H_j=\{e\}$ for  $i\neq j \, \text{and} \, i,j \in \{1,2,3\}.$ Also, $  K_i \cap K_j \neq \{e\}$ and $ H_i \cap K_j = \{e\} \, \, \text{for all} \, \, i,j \in \{1,2,3\}.$ Hence, there exists three distinct hyperedges $e_1,e_2,e_3$ such that $e_1$ contains $H_1,H_2,H_3$ and $K_1$, $e_2$ contains $H_1,H_2,H_3$ and $K_2$, and $e_3$ contains $H_1,H_2,H_3$ and $K_3$.  Hence, the incidence graph $\mathcal{I}(\tilde{\Gamma}_\mathcal{H}(\mathbb{Z}_n))$ of  $\tilde{\Gamma}_\mathcal{H}(\mathbb{Z}_n)$, contains $K_{3,3}$ and so $\tilde{\Gamma}_\mathcal{H}(\mathbb{Z}_n)$ is non-planar.
	\end{proof}
 \noindent We have the following result about orientable genus of a hypergraph to characterize $\tilde{\Gamma}_\mathcal{H}(\mathbb{Z}_n)$ whose orientable genus is 1.
	\begin{theorem}\label{incidencegraph}\cite{walsh1975hypermaps}
		For any hypergraph $\mathcal{H}$, $g(\mathcal{H})=g(\mathcal{I}(\mathcal{H}))$.
	\end{theorem}
\begin{theorem}
	The orientable genus, $g(\tilde{\Gamma}_\mathcal{H}(\mathbb{Z}_n))$ of $\tilde{\Gamma}_\mathcal{H}(\mathbb{Z}_n)$ is equal to 1  if and only if $n =p_1^{\alpha}p_2^3$ with $ 3 \leq \alpha \leq 6$ or   $n=p_1^4p_2^4$ or  $n=p_1^{\beta}p_2p_3$ with $3 \leq \beta \leq 5$, where $p_1,p_2$ and $p_3$ are distinct primes.

\end{theorem}
\begin{proof} 	
	Suppose that  $n=p_1^{\alpha_1}p_2^{\alpha_2}\ldots p_k^{\alpha_k}$, where  $k \geq 2$, $p_1,p_2, \ldots, p_k$ are distinct primes and $\alpha_1, \alpha_2, \ldots, \alpha_k$ are non-negative integers. 
	Consider the following cases:\\
	\textbf{Case 1.} Suppose that $\omega(n) =2$.\\
	\textbf{Subcase 1.1.} If $n=p_1^{\alpha_1}p_2$ or $n=p_1^{\alpha_1}p_2^2$, then by Theorem \ref{planarity} $\tilde{\Gamma}_\mathcal{H}(\mathbb{Z}_n)$ is planar and hence $g(\tilde{\Gamma}_\mathcal{H}(\mathbb{Z}_n)) \neq 1$. \\
	\textbf{Subcase 1.2.} If $n \neq p_1^{\alpha_1}p_2$ and $n \neq p_1^{\alpha_1}p_2^2$, then by Theorem \ref{planarity}, $g(\tilde{\Gamma}_\mathcal{H}(\mathbb{Z}_n))> 0$. In this subcase, we have only the following subsubcases:\\
	\textbf{Subsubcase 1.2.1.} If $n$ divides $p_1^6 p_2^3$, then  $\mathcal{I}(\tilde{\Gamma}_\mathcal{H}(\mathbb{Z}_n))$ is isomorphic to a subgraph of $\mathcal{I}(\tilde{\Gamma}_\mathcal{H}(\mathbb{Z}_{p_1^6p_2^3}))$ and  Figure \ref{torusembed1}(a) depicts the embedding of $\mathcal{I}(\tilde{\Gamma}_\mathcal{H}(\mathbb{Z}_{p_1^6p_2^3}))$ on a torus.  Hence, in this case, $\tilde{\Gamma}_\mathcal{H}(\mathbb{Z}_n)$ is embeddable on torus and therefore, $g(\tilde{\Gamma}_\mathcal{H}(\mathbb{Z}_n)) =1$.\\ 
	\noindent \textbf{Subsubcase 1.2.2.} If $n$ is divisible by $p_1^7p_2^3$, then $\mathcal{I}(\tilde{\Gamma}_\mathcal{H}(\mathbb{Z}_n))$ contains a subgraph isomorphic to  $\mathcal{I}(\tilde{\Gamma}_\mathcal{H}(\mathbb{Z}_{p_1^7p_2^3}))$ and Figure \ref{torusembed1}(b) depicts that   $\mathcal{I}(\tilde{\Gamma}_\mathcal{H}(\mathbb{Z}_{p_1^7p_2^3}))$ is not embeddable on a torus. Hence, in this case,  $\mathcal{I}(\tilde{\Gamma}_\mathcal{H}(\mathbb{Z}_n))$ is not embeddable on a torus. Therefore, $g(\tilde{\Gamma}_\mathcal{H}(\mathbb{Z}_n)) \neq 1$.	\begin{figure}[htbp!]
		\centering
		
		\subfloat[An embedding of $\mathcal{I}(\tilde{\Gamma}_\mathcal{H}(\mathbb{Z}_{p_1^6p_2^3}))$ on a torus]{\begin{tikzpicture}[scale=.7, every node/.style={circle, fill=black, minimum size=4pt, inner sep=0pt}]
				\draw[thick] (-6,-5) rectangle (7,4); 
				
				\node[label=left:\tiny{${<p_2^3>}$}] (u1) at (0,-2.5) {};
				\node[label=right:\tiny{${<p_1p_2^3>}$}] (u2) at (3,0) {};
				\node[label=right:\tiny{${<p_1^2p_2^3>}$}] (u3) at (3,1) {};
				\node[label=left:\tiny{${<p_1^6p_2^2>}$}] (u4) at (1,-1.5) {};
				\node[label=right:\tiny{${<p_1^6p_2>}$}] (u5) at (7,-2) {};
				\node[label=right:\tiny{${<p_1^6>}$}] (u6) at (7,1.5) {};
				\node[label=left:\tiny{${<p_1^6>}$}] (u7) at (-6,1.5) {};
				\node[label=left:\tiny{${<p_1^6p_2>}$}] (u8) at (-6,-2) {};
				\node[label=right:\tiny{${<p_1^3p_2^3>}$}] (u9) at (0.2,3) {};
				\node[label=above right:\tiny{${<p_1^4p_2^3>}$}] (u10) at (4,4) {};
				\node[label=below right:\tiny{${<p_1^4p_2^3>}$}] (u11) at (4,-5) {};
				\node[label=above:\tiny{${<p_1^5p_2^3>}$}] (u12) at (3,4) {};
				\node[label=below:\tiny{${<p_1^5p_2^3>}$}] (u13) at (3,-5) {};

				\node[label=below:\tiny{$e_{13}$}] (v3) at (1,-2) {};
				\node[label=below right:\tiny{$e_{14}$}] (v2) at (2,0) {};
				\node[label=left:\tiny{$e_2$}] (v1) at (5,-1) {};
				\node[label=below:\tiny{$e_9$}] (v4) at (0,1.2) {};
				\node[label=above:\tiny{$e_3$}] (v5) at (4,2) {};
				\node[label=above:\tiny{$e_7$}] (v6) at (-3,-4) {};
				\node[label=above:\tiny{$e_8$}] (v7) at (5,-2) {};
				\node[label=below:\tiny{$e_1$}] (v9) at (0,-5) {};
				\node[label=above:\tiny{$e_1$}] (v10) at (0,4) {};
				\node[label=above right:\tiny{$e_4$}] (v13) at (-1,3) {};
				\node[label=above:\tiny{$e_{10}$}] (v14) at (-1,2) {};
				\node[label=below:\tiny{$e_{16}$}] (v11) at (1,-5) {};
				\node[label=above:\tiny{$e_{16}$}] (v12) at (1,4) {};
				\node[label=right:\tiny{$e_5$}] (v15) at (5,3) {};
				\node[label=left:\tiny{$e_{11}$}] (v16) at (5,-3) {};
				\node[label=right:\tiny{$e_{17}$}] (v17) at (3,-2) {};
				\node[label=left:\tiny{$e_{15}$}] (v18) at (2,.8) {};
				\node[label=left:\tiny{$e_{18}$}] (v19) at (2.5,-3) {};
				\node[label=above:\tiny{$e_6$}] (v20) at (4,3) {};
				\node[label=below:\tiny{$e_{12}$}] (v21) at (2.2,2.8) {};
				
				
				
				\draw (u2) -- (v1);
				\draw (v1) -- (u6);
				\draw (u4) -- (v2);
				\draw (v2) -- (u2);
				\draw (u4) -- (v3);
				\draw (v3) -- (u1);
				\draw (u8) -- (v4);
				\draw (v4) -- (u3);
				\draw (u6) -- (v5);
				\draw (v5) -- (u3);
				\draw (v6) -- (u1);
				\draw (u1) -- (v9);
				\draw (u7) -- (v10);
				\draw (u7) -- (v13);
				\draw (u8) -- (v14);
				\draw (u9) -- (v13);
				\draw (u9) -- (v14);
				\draw (u10) -- (v15);
				\draw (u6) -- (v15);
				\draw (u11) -- (v16);
				\draw (u5) -- (v16);
				\draw (u11) -- (v17);
				\draw (v7) -- (u2);
				\draw (u3) -- (v18);
				\draw (u4) -- (v18);
				\draw (u12) -- (v20);
				\draw (u6) -- (v20);
				\draw (u12) -- (v21);
				\draw (u8) -- (v21);
				\draw (u13) -- (v19);

				\draw (u5) to[out=90, in=330](v7);
				\draw (u8) to[out=270, in=180](v6);
				\draw (u4) to[out=0, in=90](v11);
				\draw (u9) -- (v12);
				\draw (u4) to[out=20, in=100](v17);
				\draw (u4) to[out=10, in=130](v19);
				
				\draw[->,] (-2,4) -- (-1,4) [right] {};
				\draw[->,] (-2,-5) -- (-1,-5) [right] {};
				\draw[->,] (7,0) -- (7,1) [right] {};
				\draw[->,] (-6,0) -- (-6,1) [right] {};
		\end{tikzpicture}}
		\hspace{.2cm}
		\subfloat[$\mathcal{I}(\tilde{\Gamma}_\mathcal{H}(\mathbb{Z}_{p_1^7p_2^3}))$ on a torus]{\begin{tikzpicture}[scale=.7, every node/.style={circle, fill=black, minimum size=4pt, inner sep=0pt}]
				\draw (-6,-5) rectangle (7,4); 
				
				\node[label=left:\tiny{${<p_2^3>}$}] (u1) at (0,-2.5) {};
				\node[label=below:\tiny{${<p_1p_2^3>}$}] (u2) at (3,0) {};
				\node[label=right:\tiny{${<p_1^2p_2^3>}$}] (u3) at (3,1) {};
				\node[label=left:\tiny{${<p_1^7p_2^2>}$}] (u4) at (1,-1.5) {};
				\node[label=right:\tiny{${<p_1^7p_2>}$}] (u5) at (7,-2) {};
				\node[label=right:\tiny{${<p_1^7>}$}] (u6) at (7,1.5) {};
				\node[label=left:\tiny{${<p_1^7>}$}] (u7) at (-6,1.5) {};
				\node[label=left:\tiny{${<p_1^7p_2>}$}] (u8) at (-6,-2) {};
				\node[label=right:\tiny{${<p_1^3p_2^3>}$}] (u9) at (0.2,3) {};
				\node[label=above right:\tiny{${<p_1^4p_2^3>}$}] (u10) at (4,4) {};
				\node[label=below right:\tiny{${<p_1^4p_2^3>}$}] (u11) at (4,-5) {};
				\node[label=above:\tiny{${<p_1^5p_2^3>}$}] (u12) at (-3,1) {};
				\node[label=below:\tiny{$<p_1^6p_2^3>$}] (u13) at (3,-5) {};
				\node[label=above:\tiny{$<p_1^6p_2^3>$}] (u14) at (3,4) {};

				\node[label=below:\tiny{$e_{15}$}] (v3) at (1,-2) {};
				\node[label=below right:\tiny{$e_{16}$}] (v2) at (2,0) {};
				\node[label=left:\tiny{$e_2$}] (v1) at (5,-1) {};
				\node[label=below:\tiny{$e_{10}$}] (v4) at (0,1.4) {};
				\node[label=above:\tiny{$e_{15}$}] (v5) at (4,2) {};
				\node[label=above:\tiny{$e_8$}] (v6) at (-3,-4) {};
				\node[label=above:\tiny{$e_9$}] (v7) at (5,-2) {};
				\node[label=above:\tiny{$e_8$}] (v8) at (5,0) {};
				\node[label=below:\tiny{$e_1$}] (v9) at (0,-5) {};
				\node[label=above:\tiny{$e_1$}] (v10) at (0,4) {};
				\node[label=above:\tiny{$e_4$}] (v13) at (-1,3) {};
				\node[label=above:\tiny{$e_{11}$}] (v14) at (-1,2) {};
				\node[label=below:\tiny{$e_{18}$}] (v11) at (1,-5) {};
				\node[label=above:\tiny{$e_{18}$}] (v12) at (1,4) {};
				\node[label=right:\tiny{$e_5$}] (v15) at (5,3) {};
				\node[label=left:\tiny{$e_{12}$}] (v16) at (5,-3) {};
				\node[label=right:\tiny{$e_{19}$}] (v17) at (3,-2) {};
				\node[label=below left:\tiny{$e_6$}] (v18) at (-4,1) {};
				\node[label=left:\tiny{$e_{13}$}] (v19) at (-4.5,0) {};
				\node[label=below:\tiny{$e_{20}$}] (v20) at (2,-5) {};
				\node[label=above:\tiny{$e_{20}$}] (v21) at (2,4) {};
				\node[label=above:\tiny{$e_7$}] (v22) at (4,3) {};
				\node[label=above:\tiny{$e_{14}$}] (v23) at (0,1.8) {};
				\node[label=right:\tiny{$e_{21}$}] (v24) at (2.6,-3) {};
				\node[label=left:\tiny{$e_{17}$}] (v25) at (2,.8) {};
				
				
				
				\draw (u2) -- (v1);
				\draw (v1) -- (u6);
				\draw (u4) -- (v2);
				\draw (v2) -- (u2);
				\draw (u4) -- (v3);
				\draw (v3) -- (u1);
				\draw (u8) -- (v4);
				\draw (v4) -- (u3);
				\draw (u6) -- (v5);
				\draw (v5) -- (u3);
				\draw (v6) -- (u1);
				\draw (u6) -- (v8);
				\draw (u3) -- (v8);
				\draw (u1) -- (v9);
				\draw (u7) -- (v10);
				\draw (u7) -- (v13);
				\draw (u8) -- (v14);
				\draw (u9) -- (v13);
				\draw (u9) -- (v14);
				\draw (u10) -- (v15);
				\draw (u6) -- (v15);
				\draw (u11) -- (v16);
				\draw (u5) -- (v16);
				\draw (u11) -- (v17);
				\draw (u7) -- (v18);
				\draw (u12) -- (v18);
				\draw (u8) -- (v19);
				\draw (u12) -- (v19);
				\draw (u14) -- (v22);
				\draw (u6) -- (v22);
				\draw (u13) -- (v24);
				\draw (v7) -- (u2);
				\draw (u3) -- (v25);
				\draw (u4) -- (v25);
				
				\draw (u5) to[out=90, in=330](v7);
				\draw (u8) to[out=270, in=180](v6);
				\draw (u4) to[out=0, in=90](v11);
				\draw (u9) -- (v12);
				\draw (u4) to[out=0, in=100](v17);
				\draw (u12) to[out=30, in=330](v21);
				\draw (u4) to[out=0, in=90](v20);
				\draw (u8) to[out=30, in=200](v23);
				\draw (u14) to[out=270, in=0](v23);
				\draw (u4) to[out=0, in=100](v24);
				
				\draw[->,] (-2,4) -- (-1,4) [right] {};
				\draw[->,] (-2,-5) -- (-1,-5) [right] {};
				\draw[->,] (7,0) -- (7,1) [right] {};
				\draw[->,] (-6,0) -- (-6,1) [right] {};
		\end{tikzpicture}}
		\caption{}
		\label{torusembed1}
	\end{figure}\\
\noindent \textbf{Subsubcase 1.2.3.} If $n=p_1^4p_2^4$, then Figure \ref{torusembed2}(a) depicts an embedding of $\mathcal{I}(\tilde{\Gamma}_\mathcal{H}(\mathbb{Z}_{p_1^4p_2^4}))$ on a torus. Therefore, $g(\tilde{\Gamma}_\mathcal{H}(\mathbb{Z}_{p_1^4p_2^4}))=1$.\\
	\noindent \textbf{Subsubcase 1.2.4} If $n$ is divisible by $p_1^5p_2^4$, then $\mathcal{I}(\tilde{\Gamma}_\mathcal{H}(\mathbb{Z}_n))$ contains a subgraph isomorphic to  $\mathcal{I}(\tilde{\Gamma}_\mathcal{H}(\mathbb{Z}_{p_1^5p_2^4}))$ and Figure \ref{torusembed2}(b) depicts that  $\mathcal{I}(\tilde{\Gamma}_\mathcal{H}(\mathbb{Z}_{p_1^5p_2^4}))$ is not embeddable on a torus. Hence, in this case,  $\mathcal{I}(\tilde{\Gamma}_\mathcal{H}(\mathbb{Z}_n))$ is also not embeddable on a torus. Therefore, $g(\tilde{\Gamma}_\mathcal{H}(\mathbb{Z}_n))\neq 1$.
	\begin{figure}[htbp!]
		\centering
		\subfloat[An embedding of $\mathcal{I}(\tilde{\Gamma}_\mathcal{H}(\mathbb{Z}_{p_1^4p_2^4}))$ on a torus]{\begin{tikzpicture}[scale=.7, every node/.style={circle, fill=black, minimum size=4pt, inner sep=0pt}]
					\draw (-6,-5) rectangle (7,4); 
					
					\node[label=left:\tiny{${<p_2^4>}$}] (u1) at (0,-2.5) {};
					\node[label=right:\tiny{${<p_1p_2^4>}$}] (u2) at (7,-1) {};
					\node[label=right:\tiny{${<p_1^2p_2^4>}$}] (u3) at (3,1) {};
					\node[label=right:\tiny{${<p_1^4p_2^2>}$}] (u4) at (1,-1.5) {};
					\node[label=below:\tiny{${<p_1^4p_2>}$}] (u5) at (7,-2) {};
					\node[label=right:\tiny{${<p_1^4>}$}] (u6) at (7,1.5) {};
					\node[label=left:\tiny{${<p_1^4>}$}] (u7) at (-6,1.5) {};
					\node[label=right:\tiny{${<p_1^4p_2>}$}] (u8) at (-6,-2) {};
					\node[label=right:\tiny{${<p_1^3p_2^4>}$}] (u9) at (3,3) {};
					\node[label=right:\tiny{${<p_1^4p_2^3>}$}] (u10) at (1,2) {};
					\node[label=left:\tiny{${<p_1p_2^4>}$}] (u11) at (-6,-1) {};
					
					\node[label=below:\tiny{$e_9$}] (v3) at (1,-2) {};
					
					\node[label=above:\tiny{$e_2$}] (v19) at (-4,1) {};
					\node[label=left:\tiny{$e_{11}$}] (v1) at (2,1) {};
					5) {};
					\node[label=above:\tiny{$e_3$}] (v5) at (5,1.5) {};
					\node[label=above:\tiny{$e_5
						$}] (v6) at (-3,-4) {};
					\node[label=above:\tiny{$e_6$}] (v7) at (4,-2) {};
					\node[label=below:\tiny{$e_1$}] (v9) at (0,-5) {};
					\node[label=above:\tiny{$e_1$}] (v10) at (0,4) {};
					\node[label=above:\tiny{$e_4$}] (v13) at (0,3) {};
					\node[label=below:\tiny{$e_{12}$}] (v11) at (1,-5) {};
					\node[label=above:\tiny{$e_{12}$}] (v12) at (1,4) {};
					\node[label=left:\tiny{$e_{13}$}] (v15) at (0,0) {};
					\node[label=left:\tiny{$e_{14}$}] (v16) at (-3,1) {};
					\node[label=left:\tiny{$e_{15}$}] (v17) at (2,1.5) {};
					\node[label=right:\tiny{$e_{16}$}] (v18) at (3,2.5) {};
					\node[label=above:\tiny{$e_7$}] (v20) at (5,-5) {};
					\node[label=above:\tiny{$e_8$}] (v21) at (3,4) {};
					\node[label=below:\tiny{$e_8$}] (v22) at (3,-5) {};
					\node[label=above:\tiny{$e_{10}$}] (v23) at (2.5,0) {};
					\node[label=above:\tiny{$e_7$}] (v24) at (5,4) {};
					
					
					
					\draw (u4) -- (v1);
					\draw (v1) -- (u3);
					\draw (u4) -- (v3);
					\draw (v3) -- (u1);
					\draw (u6) -- (v5);
					\draw (v5) -- (u3);
					\draw (v6) -- (u1);
					\draw (u1) -- (v9);
					\draw (u7) -- (v10);
					\draw (u7) -- (v13);
					\draw (u9) -- (v13);
					\draw (u1) -- (v15);
					\draw (u10) -- (v15);
					\draw (u11) -- (v16);
					\draw (u10) -- (v16);
					\draw (u3) -- (v17);
					\draw (u10) -- (v17);
					\draw (u9) -- (v18);
					\draw (u10) -- (v18);
					\draw (u7) -- (v19);
					\draw (u11) -- (v19);
					\draw (u3) -- (v24);
					\draw (u5) -- (v20);
					\draw (u9) -- (v21);
					\draw (u5) -- (v22);
					\draw (u2) -- (v23);
					\draw (u4) -- (v23);
					
					\draw (u5) to[out=200, in=0](v7);
					\draw (v7) to[out=0, in=200](u2);
					\draw (u8) to[out=270, in=180](v6);
					\draw (u4) to[out=330, in=90](v11);
					\draw (u9) -- (v12);
					\draw[->,] (0,4) -- (2,4) [right] {};
					\draw[->,] (0,-5) -- (2,-5) [right] {};
					\draw[->,] (7,0) -- (7,1) [right] {};
					\draw[->,] (-6,0) -- (-6,1) [right] {};
			\end{tikzpicture}}
		\hspace{.2cm}
		\subfloat[$\mathcal{I}(\tilde{\Gamma}_\mathcal{H}(\mathbb{Z}_{p_1^5p_2^4}))$ on a torus]{\begin{tikzpicture}[scale=.7, every node/.style={circle, fill=black, minimum size=4pt, inner sep=0pt}]
					\draw (-6,-5) rectangle (7,4); 
					
					\node[label=left:\tiny{${<p_2^4>}$}] (u1) at (0,-2.5) {};
					\node[label=right:\tiny{${<p_1p_2^4>}$}] (u2) at (7,-1) {};
					\node[label=right:\tiny{${<p_1^2p_2^4>}$}] (u3) at (3,1) {};
					\node[label=right:\tiny{${<p_1^5p_2^2>}$}] (u4) at (1,-1.5) {};
					\node[label=right:\tiny{${<p_1^5p_2>}$}] (u5) at (7,-2) {};
					\node[label=right:\tiny{${<p_1^5>}$}] (u6) at (7,1.5) {};
					\node[label=left:\tiny{${<p_1^5>}$}] (u7) at (-6,1.5) {};
					\node[label=left:\tiny{${<p_1^5p_2>}$}] (u8) at (-6,-2) {};
					\node[label=right:\tiny{${<p_1^3p_2^4>}$}] (u9) at (3,3) {};
					\node[label=left:\tiny{${<p_1^5p_2^3>}$}] (u10) at (1,2) {};
					\node[label=left:\tiny{${<p_1p_2^4>}$}] (u11) at (-6,-1) {};
					\node[label=above:\tiny{${<p_1^4p_2^4>}$}] (u12) at (4.2,4) {};
					\node[label=below:\tiny{${<p_1^4p_2^4>}$}] (u13) at (4.2,-5) {};
					
					\node[label=below:\tiny{$e_{11}$}] (v3) at (1,-2) {};
					\node[label=above:\tiny{$e_2$}] (v19) at (-4,1) {};
					\node[label=left:\tiny{$e_{13}$}] (v1) at (2,1) {};
					\node[label=above:\tiny{$e_3$}] (v5) at (5,1.5) {};
					\node[label=above:\tiny{$e_6$}] (v6) at (-3,-4) {};
					\node[label=above:\tiny{$e_7$}] (v7) at (4,-2) {};
					\node[label=below:\tiny{$e_1$}] (v9) at (0,-5) {};
					\node[label=above:\tiny{$e_1$}] (v10) at (0,4) {};
					\node[label=above:\tiny{$e_4$}] (v13) at (0,3) {};
					\node[label=below:\tiny{$e_{14}$}] (v11) at (1,-5) {};
					\node[label=above:\tiny{$e_{14}$}] (v12) at (1,4) {};
					\node[label=left:\tiny{$e_{16}$}] (v15) at (0,0) {};
					\node[label=left:\tiny{$e_{17}$}] (v16) at (-3,1) {};
					\node[label=left:\tiny{$e_{18}$}] (v17) at (2,1.5) {};
					\node[label=right:\tiny{$e_{19}$}] (v18) at (3,2.5) {};
					\node[label=above:\tiny{$e_8$}] (v20) at (5,-5) {};
					\node[label=above:\tiny{$e_9$}] (v21) at (3,4) {};
					\node[label=below:\tiny{$e_9$}] (v22) at (3,-5) {};
					\node[label=above:\tiny{$e_{12}$}] (v23) at (2.5,0) {};
					\node[label=above:\tiny{$e_8$}] (v24) at (5,4) {};
					\node[label=right:\tiny{$e_{20}$}] (v25) at (3.2,2) {};
					\node[label=left:\tiny{$e_{10}$}] (v26) at (4.8,-4) {};
					\node[label=left:\tiny{$e_{15}$}] (v27) at (3,-3) {};
					\node[label=above:\tiny{$e_5$}] (v28) at (5.5,2.8) {};
					
					
					
					\draw (u4) -- (v1);
					\draw (v1) -- (u3);
					\draw (u4) -- (v3);
					\draw (v3) -- (u1);
					\draw (u6) -- (v5);
					\draw (v5) -- (u3);
					\draw (v6) -- (u1);
					\draw (u1) -- (v9);
					\draw (u7) -- (v10);
					\draw (u7) -- (v13);
					\draw (u9) -- (v13);
					\draw (u1) -- (v15);
					\draw (u10) -- (v15);
					\draw (u11) -- (v16);
					\draw (u10) -- (v16);
					\draw (u3) -- (v17);
					\draw (u10) -- (v17);
					\draw (u9) -- (v18);
					\draw (u10) -- (v18);
					\draw (u7) -- (v19);
					\draw (u11) -- (v19);
					\draw (u3) -- (v24);
					\draw (u5) -- (v20);
					\draw (u9) -- (v21);
					\draw (u5) -- (v22);
					\draw (u2) -- (v23);
					\draw (u4) -- (v23);
					\draw (u10) -- (v25);
					\draw (u12) -- (v25);
					\draw (u13) -- (v26);
					\draw (u5) -- (v26);
					\draw (u13) -- (v27);
					\draw (u4) -- (v27);
					\draw (u12) -- (v28);
					\draw (u6) -- (v28);
					
					\draw (u5) to[out=200, in=0](v7);
					\draw (v7) to[out=0, in=200](u2);
					\draw (u8) to[out=270, in=180](v6);
					\draw (u4) to[out=330, in=90](v11);
					\draw (u9) -- (v12);
					\draw[->,] (0,4) -- (2,4) [right] {};
					\draw[->,] (0,-5) -- (2,-5) [right] {};
					\draw[->,] (7,0) -- (7,1) [right] {};
					\draw[->,] (-6,0) -- (-6,1) [right] {};
			\end{tikzpicture}}
		\caption{}
		\label{torusembed2}
	\end{figure}\\
	\noindent\textbf{Case 2.} Suppose $\omega(n)  = 3$.\\
\textbf{Subcase 2.1.} If $n=p_1p_2p_3$ or $n=p_1^2p_2p_3$, then by Theorem \ref{planarity}, $\tilde{\Gamma}_\mathcal{H}(\mathbb{Z}_n)$ is planar and hence, $g(\tilde{\Gamma}_\mathcal{H}(\mathbb{Z}_n)) \neq 1$.\\
\textbf{Subcase 2.2.}  If $n \neq p_1p_2p_3$ or $n \neq p_1^2p_2p_3$, then we have the following cases:\\
 \noindent \textbf{Subsubcase 2.2.1.} If $n$ divides $p_1^5p_2p_3$, then $\mathcal{I}(\tilde{\Gamma}_\mathcal{H}(\mathbb{Z}_n))$ is  isomorphic to a subgraph of  $\mathcal{I}(\tilde{\Gamma}_\mathcal{H}(\mathbb{Z}_{p_1^5p_2p_3}))$ and Figure \ref{torusembed3}(a)  depicts an embedding of $\mathcal{I}(\tilde{\Gamma}_\mathcal{H}(\mathbb{Z}_{p_1^5p_2p_3}))$  on a torus.  Hence, in this case,  $\mathcal{I}(\tilde{\Gamma}_\mathcal{H}(\mathbb{Z}_n))$ is also embeddable on a torus. Therefore, $g(\tilde{\Gamma}_\mathcal{H}(\mathbb{Z}_n)) = 1$.\\
 \textbf{Subsubcase 2.2.2.} If $n$ is divisible by  $p_1^6p_2p_3$, then $\mathcal{I}(\tilde{\Gamma}_\mathcal{H}(\mathbb{Z}_n))$ contains a subgraph isomorphic to $\mathcal{I}(\tilde{\Gamma}_\mathcal{H}(\mathbb{Z}_{p_1^6p_2p_3}))$ and Figure \ref{torusembed3}(b) depicts that  $\mathcal{I}(\tilde{\Gamma}_\mathcal{H}(\mathbb{Z}_{p_1^6p_2p_3}))$  is not embeddable on a torus. Hence, in this case,  $\mathcal{I}(\tilde{\Gamma}_\mathcal{H}(\mathbb{Z}_n))$ is also not embeddable on a torus and therefore, $g(\tilde{\Gamma}_\mathcal{H}(\mathbb{Z}_n)) \neq  1$.
 \begin{figure}[htbp!]
 	\centering
 \subfloat[An embedding of $\mathcal{I}(\tilde{\Gamma}_\mathcal{H}(\mathbb{Z}_{p_1^5p_2p_3}))$  on a torus]{\begin{tikzpicture}[scale=1.2, every node/.style={circle, fill=black, minimum size=4pt, inner sep=0pt}]
 				\draw (-5,.5) rectangle (3,7); 
 				
 				\node[label=left:\tiny{$<p_1^5p_2>$}] (u1) at (-5,5) {};
 				\node[label=left:\tiny{$<p_1^5p_3>$}] (u2) at (-5,3.5) {};
 				\node[label=left:\tiny{$<p_2p_3>$}] (u3) at (2.3,5) {};
 				\node[label=right:\tiny{$<p_1p_2p_3>$}] (u4) at (-3,5) {};
 				\node[label=left:\tiny{$<p_3>$}] (u5) at (-4,6) {};
 				\node[label=above left:\tiny{$<p_1p_3>$}] (u6) at (-3.5,6) {};
 				\node[label=left:\tiny{$<p_2>$}] (u7) at (-4.3,2) {};
 				\node[label=below:\tiny{$<p_1p_2>$}] (u8) at (-3.8,2) {};
 				\node[label=above:\tiny{$<p_1^5>$}] (u9) at (0,7) {};
 				\node[label=left:\tiny{$<p_1^2p_2p_3>$}] (u10) at (0.5,3) {};
 				\node[label=right:\tiny{$<p_1^2p_3>$}] (u11) at (-3,6) {};
 				\node[label=right:\tiny{$<p_1^2p_2>$}] (u12) at (-3,2) {};
 				\node[label=right:\tiny{$<p_1^5p_3>$}] (u14) at (3,3.5) {};
 				\node[label=right:\tiny{$<p_1^5p_2>$}] (u15) at (3,5) {};
 				\node[label=right:\tiny{$<p_1^3p_3>$}] (u16) at (-2,5.8) {};
 				\node[label=right:\tiny{$<p_1^3p_2>$}] (u17) at (-1.5,2) {};
 				\node[label=below:\tiny{$<p_1^5>$}] (u18) at (0,.5) {};
 				\node[label=below:\tiny{$<p_1^3p_2p_3>$}] (u19) at (1,6) {};
 				\node[label=right:\tiny{$<p_1^4p_2>$}] (u20) at (-1,2.5) {};
 				\node[label=right:\tiny{$<p_1^4p_3>$}] (u21) at (-1,5.8) {};
 				\node[label=left:\tiny{$<p_1^4p_2p_3>$}] (u22) at (-2,6.5) {};
 				\node[label=below:\tiny{$e_1$}] (v1) at (2,4) {};
 				\node[label=below:\tiny{$e_2$}] (v2) at (-3.5,5) {};
 				\node[label=right:\tiny{$e_3$}] (v3) at (1,4) {};
 				\node[label=below:\tiny{$e_4$}] (v4) at (1,.5) {};
 				\node[label=left:\tiny{$e_5$}] (v5) at (-4.5,5.5) {};
 				\node[label=left:\tiny{$e_6$}] (v6) at (-4,5.5) {};
 				\node[label=left:\tiny{$e_7$}] (v7) at (-3.5,5.5) {};
 				\node[label=above:\tiny{$e_8$}] (v8) at (-3,5.3) {};
 				\node[label=left:\tiny{$e_9$}] (v9) at (-4.5,3) {};
 				\node[label=left:\tiny{$e_{10}$}] (v10) at (-3.8,3) {};
 				\node[label=left:\tiny{$e_{11}$}] (v11) at (-3,3) {};
 				\node[label=right:\tiny{$e_{12}$}] (v12) at (-2.5,3) {};
 				\node[label=below:\tiny{$e_{13}$}] (v13) at (0.5,6) {};
 				\node[label=right:\tiny{$e_{14}$}] (v14) at (1,2) {};
 				\node[label=right:\tiny{$e_{15}$}] (v15) at (0,4.5) {};
 				\node[label=right:\tiny{$e_{16}$}] (v16) at (1,3) {};
 				\node[label=above:\tiny{$e_4$}] (v17) at (1,7) {};
 				\node[label=right:\tiny{$e_{17}$}] (v18) at (-1.8,3) {};
 				\node[label=above:\tiny{$e_{18}$}] (v19) at (-2,5.3) {};
 				\node[label=above:\tiny{$e_{19}$}] (v20) at (-2,7) {};
 				\node[label=below:\tiny{$e_{19}$}] (v21) at (-2,.5) {};
 				\node[label=below:\tiny{$e_{20}$}] (v22) at (-.5,6.5) {};
 				
 				\draw (u15) -- (v1);
 				\draw (u14) -- (v1);
 				\draw (u3) -- (v1);
 				\draw (u1) -- (v2);
 				\draw (u2) -- (v2);
 				\draw (u4) -- (v2);
 				\draw (u2) -- (v3);
 				\draw (u10) -- (v3);
 				\draw (u15) -- (v17);
 				\draw (u19) -- (v17);
 				\draw (u14) -- (v4);
 				\draw (u1) -- (v5);
 				\draw (u5) -- (v5);
 				\draw (u1) -- (v6);
 				\draw (u6) -- (v6);
 				\draw (u1) -- (v7);
 				\draw (u11) -- (v7);
 				\draw (u1) -- (v8);
 				\draw (u16) -- (v8);
 				\draw (u7) -- (v9);
 				\draw (u2) -- (v9);
 				\draw (u8) -- (v10);
 				\draw (u2) -- (v10);
 				\draw (u12) -- (v11);
 				\draw (u2) -- (v11);
 				\draw (u17) -- (v12);
 				\draw (u2) -- (v12);
 				\draw (u19) -- (v13);
 				\draw (u9) -- (v13);
 				\draw (u10) -- (v14);
 				\draw (u18) -- (v14);
 				\draw (u4) -- (v15);
 				\draw (u9) -- (v15);
 				\draw (u3) -- (v16);
 				\draw (u2) -- (v18);
 				\draw (u20) -- (v18);
 				\draw (u1) -- (v19);
 				\draw (u22) -- (v20);
 				\draw (u21) -- (v19);
 				\draw (u22) -- (v22);
 				\draw (u9) -- (v22);
 				\draw (u15) to[out=150, in=90](v3);
 				\draw (u18) to[out=10, in=290](v16);
 				\draw (u1) to[out=90, in=180](v20);
 				\draw (u2) to[out=270, in=180](v21);
 				\draw[->,] (0,7) -- (2,7) [right] {};
 				\draw[->,] (0,.5) -- (2,.5) [right] {};
 				\draw[->,] (3,3) -- (3,4) [right] {};
 				\draw[->,] (-5,3) -- (-5,4) [right] {};
 				
 		\end{tikzpicture}}
 	\hspace{.2cm}
 		\subfloat[$\mathcal{I}(\tilde{\Gamma}_\mathcal{H}(\mathbb{Z}_{p_1^6p_2p_3}))$ on a torus]{\begin{tikzpicture}[scale=1.2, every node/.style={circle, fill=black, minimum size=4pt, inner sep=0pt}]
 				\draw (-5,.5) rectangle (3,7); 
 				
 				\node[label=left:\tiny{$<p_1^6p_2>$}] (u1) at (-5,5) {};
 				\node[label=left:\tiny{$<p_1^6p_3>$}] (u2) at (-5,3.5) {};
 				\node[label=left:\tiny{$<p_2p_3>$}] (u3) at (2.3,5) {};
 				\node[label=right:\tiny{$<p_1p_2p_3>$}] (u4) at (-3,5) {};
 				\node[label=left:\tiny{$<p_3>$}] (u5) at (-4,6) {};
 				\node[label=above left:\tiny{$<p_1p_3>$}] (u6) at (-3.5,6) {};
 				\node[label=left:\tiny{$<p_2>$}] (u7) at (-4.3,2) {};
 				\node[label=below:\tiny{$<p_1p_2>$}] (u8) at (-3.8,2) {};
 				\node[label=above:\tiny{$<p_1^5>$}] (u9) at (0,7) {};
 				\node[label=left:\tiny{$<p_1^2p_2p_3>$}] (u10) at (0.5,3) {};
 				\node[label=right:\tiny{$<p_1^2p_3>$}] (u11) at (-3,6) {};
 				\node[label=right:\tiny{$<p_1^2p_2>$}] (u12) at (-3,2) {};
 				\node[label=right:\tiny{$<p_1^6p_3>$}] (u14) at (3,3.5) {};
 				\node[label=right:\tiny{$<p_1^6p_2>$}] (u15) at (3,5) {};
 				\node[label=right:\tiny{$<p_1^3p_3>$}] (u16) at (-2,5.8) {};
 				\node[label=right:\tiny{$<p_1^3p_2>$}] (u17) at (-1.5,2) {};
 				\node[label=below:\tiny{$<p_1^5>$}] (u18) at (0,.5) {};
 				\node[label=below:\tiny{$<p_1^3p_2p_3>$}] (u19) at (1,6) {};
 				\node[label=right:\tiny{$<p_1^4p_2>$}] (u20) at (-1,2.5) {};
 				\node[label=right:\tiny{$<p_1^4p_3>$}] (u21) at (-1,5.8) {};
 				\node[label=left:\tiny{$<p_1^4p_2p_3>$}] (u22) at (-2,6.5) {};
 				\node[label=right:\tiny{$<p_1^5p_3>$}] (u23) at (-1,5.5) {};
 				\node[label=right:\tiny{$<p_1^5p_2>$}] (u24) at (-1.5,3.3) {};
 				\node[label=left:\tiny{$<p_1^5p_2p_3>$}] (u25) at (-1.5,1) {};
 				
 				\node[label=below:\tiny{$e_1$}] (v1) at (2,4) {};
 				\node[label=below:\tiny{$e_2$}] (v2) at (-3.5,5) {};
 				\node[label=right:\tiny{$e_3$}] (v3) at (1,4) {};
 				\node[label=below:\tiny{$e_4$}] (v4) at (1,.5) {};
 				\node[label=left:\tiny{$e_5$}] (v5) at (-4.5,5.5) {};
 				\node[label=left:\tiny{$e_6$}] (v6) at (-4,5.5) {};
 				\node[label=left:\tiny{$e_7$}] (v7) at (-3.5,5.5) {};
 				\node[label=above:\tiny{$e_8$}] (v8) at (-3,5.3) {};
 				\node[label=left:\tiny{$e_9$}] (v9) at (-4.5,3) {};
 				\node[label=left:\tiny{$e_{10}$}] (v10) at (-3.8,3) {};
 				\node[label=left:\tiny{$e_{11}$}] (v11) at (-3,3) {};
 				\node[label=right:\tiny{$e_{12}$}] (v12) at (-2.5,3) {};
 				\node[label=below:\tiny{$e_{13}$}] (v13) at (0.5,6) {};
 				\node[label=right:\tiny{$e_{14}$}] (v14) at (1,2) {};
 				\node[label=right:\tiny{$e_{15}$}] (v15) at (0,4.5) {};
 				\node[label=right:\tiny{$e_{16}$}] (v16) at (1,3) {};
 				\node[label=above:\tiny{$e_4$}] (v17) at (1,7) {};
 				\node[label=right:\tiny{$e_{17}$}] (v18) at (-1.8,3) {};
 				\node[label=above:\tiny{$e_{18}$}] (v19) at (-2,5.3) {};
 				\node[label=above:\tiny{$e_{19}$}] (v20) at (-2,7) {};
 				\node[label=below:\tiny{$e_{19}$}] (v21) at (-2,.5) {};
 				\node[label=below:\tiny{$e_{20}$}] (v22) at (-.5,6.5) {};
 				\node[label=above:\tiny{$e_{21}$}] (v23) at (-1.3,5.2) {};
 				\node[label=above:\tiny{$e_{22}$}] (v24) at (-2,3.3) {};
 				\node[label=below:\tiny{$e_{23}$}] (v25) at (-1,.5) {};
 				\node[label=above:\tiny{$e_{23}$}] (v26) at (-1,7) {};
 				\node[label=above:\tiny{$e_{24}$}] (v27) at (-1,1) {};
 				
 				\draw (u15) -- (v1);
 				\draw (u14) -- (v1);
 				\draw (u3) -- (v1);
 				\draw (u1) -- (v2);
 				\draw (u2) -- (v2);
 				\draw (u4) -- (v2);
 				\draw (u2) -- (v3);
 				\draw (u10) -- (v3);
 				\draw (u15) -- (v17);
 				\draw (u19) -- (v17);
 				\draw (u14) -- (v4);
 				\draw (u1) -- (v5);
 				\draw (u5) -- (v5);
 				\draw (u1) -- (v6);
 				\draw (u6) -- (v6);
 				\draw (u1) -- (v7);
 				\draw (u11) -- (v7);
 				\draw (u1) -- (v8);
 				\draw (u16) -- (v8);
 				\draw (u7) -- (v9);
 				\draw (u2) -- (v9);
 				\draw (u8) -- (v10);
 				\draw (u2) -- (v10);
 				\draw (u12) -- (v11);
 				\draw (u2) -- (v11);
 				\draw (u17) -- (v12);
 				\draw (u2) -- (v12);
 				\draw (u19) -- (v13);
 				\draw (u9) -- (v13);
 				\draw (u10) -- (v14);
 				\draw (u18) -- (v14);
 				\draw (u4) -- (v15);
 				\draw (u9) -- (v15);
 				\draw (u3) -- (v16);
 				\draw (u2) -- (v18);
 				\draw (u20) -- (v18);
 				\draw (u1) -- (v19);
 				\draw (u22) -- (v20);
 				\draw (u21) -- (v19);
 				\draw (u22) -- (v22);
 				\draw (u9) -- (v22);
 				\draw (u1) -- (v23);
 				\draw (u23) -- (v23);
 				\draw (u2) -- (v24);
 				\draw (u24) -- (v24);
 				\draw (u25) -- (v25);
 				\draw (u15)  -- (v26);
 				\draw (u25) -- (v27);
 				\draw (u18) -- (v27);
 				\draw (u15) to[out=150, in=90](v3);
 				\draw (u18) to[out=10, in=290](v16);
 				\draw (u1) to[out=90, in=180](v20);
 				\draw (u2) to[out=270, in=180](v21);
 				\draw (u2) to[out=275, in=180](v25);
 				
 					\draw[->,] (0,7) -- (2,7) [right] {};
 				\draw[->,] (0,.5) -- (2,.5) [right] {};
 				\draw[->,] (3,3) -- (3,4) [right] {};
 				\draw[->,] (-5,3) -- (-5,4) [right] {};
 				
 		\end{tikzpicture}}
 
 	\caption{}
 	\label{torusembed3}
 \end{figure}\\
\noindent \textbf{Subsubcase 2.2.3.}  If $n$ contain atleast two prime divisors with power greater than  or equal to 2, then  $\mathcal{I}(\tilde{\Gamma}_\mathcal{H}(\mathbb{Z}_n))$ contains a subgraph isomorphic to $\mathcal{I}(\tilde{\Gamma}_\mathcal{H}(\mathbb{Z}_{p_1^2p_2^2p_3}))$ and Figure \ref{torusembed4}(a) depicts that $\mathcal{I}(\tilde{\Gamma}_\mathcal{H}(\mathbb{Z}_{p_1^2p_2^2p_3}))$  is not embeddable on a torus. Hence, in this case,  $\mathcal{I}(\tilde{\Gamma}_\mathcal{H}(\mathbb{Z}_n))$ is also not embeddable on a torus and therefore, $g(\tilde{\Gamma}_\mathcal{H}(\mathbb{Z}_n)) \neq  1$.\\
	\noindent\textbf{Case 3.} Suppose $\omega(n) > 3$.\\
Observe that $\mathcal{I}(\tilde{\Gamma}_\mathcal{H}(\mathbb{Z}_n))$ contains a subgraph isomorphic to $\mathcal{I}(\tilde{\Gamma}_\mathcal{H}(\mathbb{Z}_{p_1p_2p_3p_4}))$ and Figure \ref{torusembed4}(b) depicts that  $\mathcal{I}(\tilde{\Gamma}_\mathcal{H}(\mathbb{Z}_{p_1p_2p_3p_4}))$ is not embeddable on a torus. Hence, in this case,  $\mathcal{I}(\tilde{\Gamma}_\mathcal{H}(\mathbb{Z}_n))$ is also not embeddable on a torus and therefore, $g(\tilde{\Gamma}_\mathcal{H}(\mathbb{Z}_n)) \neq 1$.
 \begin{figure}[htbp!]
	\centering
	\subfloat[$\mathcal{I}(\tilde{\Gamma}_\mathcal{H}(\mathbb{Z}_{p_1^2p_2^2p_3}))$ on torus]{\begin{tikzpicture}[scale=1.2, every node/.style={circle, fill=black, minimum size=4pt, inner sep=0pt}]
				\draw (-5,.5) rectangle (3,7); 
				
				\node[label=above:\tiny{$<p_1^2p_2^2>$}] (u1) at (-1,7) {};
				\node[label=right:\tiny{$<p_1^2p_2p_3>$}] (u2) at (-3,3.5) {};
				\node[label=right:\tiny{$<p_1p_2^2p_3>$}] (u3) at (3,5) {};
				\node[label=left:\tiny{$<p_2^2p_3>$}] (u4) at (-2.5,5) {};
				\node[label=left:\tiny{$<p_1^2p_3>$}] (u5) at (1,5.9) {};
				\node[label=right:\tiny{$<p_1p_2p_3>$}] (u6) at (0,3) {};
				\node[label=right:\tiny{$<p_1p_3>$}] (u7) at (.2,2.1) {};
				\node[label=right:\tiny{$<p_2p_3>$}] (u8) at (.2,1) {};
				\node[label=below:\tiny{$<p_3>$}] (u9) at (-2,6.5) {};
				\node[label=right:\tiny{$<p_1p_2^2>$}] (u11) at (1.4,4) {};
				\node[label=above:\tiny{$<p_1^2p_2>$}] (u12) at (-4.5,6) {};
				\node[label=above:\tiny{$<p_1^2>$}] (u13) at (-3.5,6.2) {};
				\node[label=below:\tiny{$<p_1^2p_2^2>$}] (u14) at (-1,.5) {};
				\node[label=left:\tiny{$<p_1p_2^2p_3>$}] (u15) at (-5,5) {};
				\node[label=right:\tiny{$<p_2^2>$}] (u16) at (-2,2.8) {};
				
				\node[label=left:\tiny{$e_1$}] (v1) at (-4,3.5) {};
				\node[label=left:\tiny{$e_2$}] (v2) at (-3,4.5) {};
				\node[label=above:\tiny{$e_3$}] (v3) at (1,6.4) {};
				\node[label=below:\tiny{$e_4$}] (v4) at (-1,6.5) {};
				\node[label=below right:\tiny{$e_5$}] (v5) at (-0.5,3) {};
				\node[label=below right:\tiny{$e_6$}] (v6) at (-0.5,2) {};
				\node[label=below:\tiny{$e_7$}] (v7) at (-0.5,1) {};
				\node[label=right:\tiny{$e_8$}] (v8) at (-1.5,6.5) {};
				\node[label=below:\tiny{$e_9$}] (v9) at (0,4.5) {};
				\node[label=left:\tiny{$e_{10}$}] (v10) at (-3,1) {};
				\node[label=left:\tiny{$e_{11}$}] (v11) at (-4.3,5.5) {};
				\node[label=right:\tiny{$e_{12}$}] (v12) at (-4,5.5) {};
				\node[label=right:\tiny{$e_{13}$}] (v13) at (0,5) {};
				\node[label=right:\tiny{$e_{14}$}] (v14) at (2,5) {};
				\node[label=left:\tiny{$e_{15}$}] (v15) at (0,5.5) {};

				
				\draw (u1) to[out=180, in=90] (v1);
				\draw (u2) -- (v1);
				\draw (u15) -- (v1);
				\draw (u14) -- (v2);
				\draw (u2) -- (v2);
				\draw (u4) -- (v2);
				\draw (u1) -- (v3);
				\draw (u3) -- (v3);
				\draw (u5) -- (v3);
				\draw (u1) -- (v4);
				\draw (u4) -- (v4);
				\draw (u5) -- (v4);
				\draw (u6) -- (v5);
				\draw (u14) -- (v5);
				\draw (u7) -- (v6);
				\draw (u14) -- (v6);
				\draw (u8) -- (v7);
				\draw (u14) -- (v7);
				\draw (u1) -- (v8);
				\draw (u9) -- (v8);
				\draw (u2) -- (v9);
				\draw (u11) -- (v9);
				\draw (u2) -- (v10);
				\draw (u16) -- (v10);
				\draw (u15) -- (v11);
				\draw (u12) -- (v11);
				\draw (u15) -- (v12);
				\draw (u13) -- (v12);
				\draw (u11) -- (v13);
				\draw (u4) -- (v13);
				\draw (u5) -- (v14);
				\draw (u11) -- (v14);
				\draw (u5) -- (v15);
				\draw (u16) -- (v15);
				\draw[->,] (-1,7) -- (0,7) [right] {};
				\draw[->,] (-1,.5) -- (0,.5) [right] {};
				\draw[->,] (3,3) -- (3,4) [right] {};
				\draw[->,] (-5,3) -- (-5,4) [right] {};

		\end{tikzpicture}}
	\hspace{.2cm}
	\subfloat[$\mathcal{I}(\tilde{\Gamma}_\mathcal{H}(\mathbb{Z}_{p_1p_2p_3p_4}))$ on torus]{\begin{tikzpicture}[scale=1.2, every node/.style={circle, fill=black, minimum size=4pt, inner sep=0pt}]
				\draw (-3,-2) rectangle (5,4); 
				
				\node[label=above:\tiny{${<p_1p_2p_3>}$}] (u1) at (2,4) {};
				\node[label=below right:\tiny{${<p_1p_2p_4>}$}] (u2) at (1,3) {};
				\node[label=right:\tiny{${<p_1p_3p_4>}$}] (u3) at (-1,2.5) {};
				\node[label=right:\tiny{${<p_2p_3p_4>}$}] (u4) at (-1,0.4) {};
				\node[label=above:\tiny{${<p_3p_4>}$}] (u5) at (3,4) {};
				\node[label=left:\tiny{${<p_2p_4>}$}] (u6) at (-.8,3) {};
				\node[label=below:\tiny{${<p_1p_2p_3>}$}] (u7) at (2,-2) {};
				\node[label=left:\tiny{${<p_1p_4>}$}] (u8) at (-1.8,-1.5) {};
				\node[label=right:\tiny{${<p_4>}$}] (u9) at (2.8,-1) {};
				\node[label=left:\tiny{${<p_2p_3>}$}] (u10) at (-1.8,3.5) {};
				\node[label=right:\tiny{${<p_1p_3>}$}] (u11) at (3,0.2) {};
				\node[label=below:\tiny{${<p_3>}$}] (u12) at (2.5,1) {};
				\node[label=right:\tiny{${<p_1p_2>}$}] (u13) at (-2.5,1) {};
				\node[label=right:\tiny{${<p_2>}$}] (u14) at (0,1.5) {};
				\node[label=left:\tiny{${<p_3>}$}] (u15) at (-.5,1) {};
				\node[label=below:\tiny{${<p_3p_4>}$}] (u16) at (3,-2) {};

				\node[label=below:\tiny{$e_1$}] (v1) at (2,1) {};
				\node[label=right:\tiny{$e_2$}] (v2) at (3,3) {};
				\node[label=below:\tiny{$e_3$}] (v3) at (-.2,3) {};
				\node[label=below:\tiny{$e_4$}] (v4) at (0,-.8) {};
				\node[label=right:\tiny{$e_5$}] (v5) at (2.3,-.8) {};
				\node[label=left:\tiny{$e_6$}] (v6) at (-3,2.5) {};
				\node[label=right:\tiny{$e_6$}] (v7) at (5,2.5) {};
				\node[label=right:\tiny{$e_7$}] (v8) at (3.5,2) {};
				\node[label=right:\tiny{$e_8$}] (v9) at (2.5,2) {};
				\node[label=left:\tiny{$e_9$}] (v10) at (-2,2) {};
				\node[label=left:\tiny{$e_{10}$}] (v11) at (-1,1.5) {};
				\node[label=right:\tiny{$e_{11}$}] (v12) at (0.5,1) {};
				\node[label=left:\tiny{$e_{12}$}] (v13) at (-3,0.3) {};
				\node[label=below:\tiny{$e_{13}$}] (v14) at (1,-2) {};
				\node[label=below:\tiny{$e_{14}$}] (v15) at (0,-2) {};
				\node[label=above:\tiny{$e_{14}$}] (v16) at (0,4) {};
				\node[label=above:\tiny{$e_{13}$}] (v17) at (1,4) {};
				\node[label=right:\tiny{$e_{12}$}] (v18) at (5,0.3) {};

				
				
				
				\draw (u1) -- (v1);
				\draw (u2) -- (v1);
				\draw (u3) -- (v1);
				\draw (u4) -- (v1);
				\draw (u1) -- (v2);
				\draw (u2) -- (v2);
				\draw (u5) -- (v2);
				\draw (u1) -- (v3);
				\draw (u3) -- (v3);
				\draw (u6) -- (v3);
				\draw (u7) -- (v4);
				\draw (u4) -- (v4);
				\draw (u8) -- (v4);
				\draw (u9) -- (v5);
				\draw (u7) -- (v5);
				\draw (u2) -- (v7);
				\draw (u3) -- (v6);
				\draw (u10) -- (v6);
				\draw (u2) -- (v8);
				\draw (u11) -- (v8);
				\draw (u2) -- (v9);
				\draw (u12) -- (v9);
				\draw (u3) -- (v10);
				\draw (u4) -- (v10);
				\draw (u13) -- (v10);
				\draw (u3) -- (v11);
				\draw (u14) -- (v11);
				\draw (u4) -- (v12);
				\draw (u15) -- (v12);
				\draw (u13) -- (v13);
				\draw (u11) -- (v14);
				\draw (u6) -- (v17);
				\draw (u8) -- (v15);
				\draw (u10) -- (v16);
				\draw (u16) -- (v18);


				\draw (u4) to[out=330, in=250](v8);
				\draw[->,] (-2,4) -- (-1,4) [right] {};
				\draw[->,] (-2,-2) -- (-1,-2) [right] {};
				\draw[->,] (5,0) -- (5,1) [right] {};
				\draw[->,] (-3,0) -- (-3,1) [right] {};
		\end{tikzpicture}}
	\caption{}
	\label{torusembed4}
\end{figure}
		\end{proof}
		\begin{corollary}
			 $\tilde{\Gamma}_\mathcal{H}(\mathbb{Z}_n)$ is toroidal if and only if $n=p_1^{\alpha_1}p_2^{\alpha}$ with $\alpha=1,2$ or $n=p_1^{\beta}p_2^3$ with $3 \leq \beta \leq 6$ or $n=p_1^{\gamma}p_2p_3$ with $1 \leq \gamma \leq 5$, where $p_1,p_2,p_3$ are distinct primes and $\alpha_1$ is a positive integer.
		\end{corollary}

		\noindent We have the following result about non-orientable genus of a hypergraph to characterize  $\tilde{\Gamma}_\mathcal{H}(\mathbb{Z}_n)$ whose non-orientable genus is 1.
		
		\begin{theorem}\label{incidencegraph2}\cite{walsh1975hypermaps}
			For any hypergraph $\mathcal{H}$, $\tilde{g}(\mathcal{H})=\tilde{g}(\mathcal{I}(\mathcal{H}))$.
		\end{theorem}
		\begin{theorem}
			The non-orientable genus, $\tilde{g}(\tilde{\Gamma}_\mathcal{H}(\mathbb{Z}_n))$, of $\tilde{\Gamma}_\mathcal{H}(\mathbb{Z}_n)$ is equal to 1 if and only if  $n=p_1^{\alpha_1}p_2^3$ with $\alpha_1=3,4$ or $n=p_1^{\alpha_1}p_2p_3$ with $\alpha_1=3,4$, where $p_1,p_2$ and $p_3$ are distinct primes.
			
		\end{theorem}
		
		\begin{proof}
				Suppose that  $n=p_1^{\alpha_1}p_2^{\alpha_2}\ldots p_k^{\alpha_k}$, where  $k \geq 2$, $p_1,p_2, \ldots, p_k$ are distinct primes,  $\alpha_1, \alpha_2, \ldots, \alpha_k$ are non-negative integers. 
		Consider the following cases:\\
			\textbf{Case 1.} Suppose that $\omega(n) =2$.\\
			\textbf{Subcase 1.1.} If $n=p_1^{\alpha_1}p_2$ or $n=p_1^{\alpha_1}p_2^2$, then, by Theorem \ref{planarity}, $\tilde{\Gamma}_\mathcal{H}(\mathbb{Z}_n)$ is planar and hence $\tilde{g}(\tilde{\Gamma}_\mathcal{H}(\mathbb{Z}_n)) \neq 1$. \\
			\textbf{Subcase 1.2.} If $n \neq p_1^{\alpha_1}p_2$ and  $n \neq p_1^{\alpha_1}p_2^2$,  then, by Theorem \ref{planarity}, $\tilde{g}(\tilde{\Gamma}_\mathcal{H}(\mathbb{Z}_n))> 0$. In this subcase, we have the only following subsubcases:\\
			\textbf{Subsubcase 1.2.(a)} If $n$ divides $p_1^4p_2^3$, then $\mathcal{I}(\tilde{\Gamma}_\mathcal{H}(\mathbb{Z}_n))$ is isomorphic to a subgraph of $\mathcal{I}(\tilde{\Gamma}_\mathcal{H}(\mathbb{Z}_{p_1^4p_2^3}))$ and Figure \ref{proj1}(a) depicts an embedding of $\mathcal{I}(\tilde{\Gamma}_\mathcal{H}(\mathbb{Z}_{p_1^4p_2^3}))$ on a projective plane. Hence, in this case,  $\tilde{g}(\tilde{\Gamma}_\mathcal{H}(\mathbb{Z}_n)) = 1$.\\
			\textbf{Subsubcase 1.2.(b)} If $n$ is divisible by $p_1^5p_2^3$, then $\mathcal{I}(\tilde{\Gamma}_\mathcal{H}(\mathbb{Z}_n))$ contains a subgraph isomorphic to $\mathcal{I}(\tilde{\Gamma}_\mathcal{H}(\mathbb{Z}_{p_1^5p_2^3}))$ and  Figure \ref{proj1}(b) depicts that  $\mathcal{I}(\tilde{\Gamma}_\mathcal{H}(\mathbb{Z}_{p_1^5p_2^3}))$ is not embeddable on a projective plane. Hence, in this case, $\tilde{g}(\tilde{\Gamma}_\mathcal{H}(\mathbb{Z}_n)) \neq  1$.\\
			\textbf{Subsubcase 1.2.(c)} If $n$ is divisible by  $p_1^4p_2^4$, then $\mathcal{I}(\tilde{\Gamma}_\mathcal{H}(\mathbb{Z}_n))$ contains a subgraph isomorphic to $\mathcal{I}(\tilde{\Gamma}_\mathcal{H}(\mathbb{Z}_{p_1^4p_2^4}))$ and Figure \ref{proj1}(c) depicts that  $\mathcal{I}(\tilde{\Gamma}_\mathcal{H}(\mathbb{Z}_{p_1^4p_2^4}))$ is not embeddable on a projective plane. Hence, in this case,  $\tilde{g}(\tilde{\Gamma}_\mathcal{H}(\mathbb{Z}_n)) \neq  1$. 
			
			 \begin{figure}[htbp!]
			 	\centering
			\subfloat[An embedding of $\mathcal{I}(\tilde{\Gamma}_\mathcal{H}(\mathbb{Z}_{p_1^4p_2^3}))$ on a projective plane]{
					\begin{tikzpicture}[scale=.65, every node/.style={circle, fill=black, minimum size=4pt, inner sep=0pt}]
						\draw (-3,-5) rectangle (7,4); 
						
						\node[label=left:\tiny{${<p_2^3>}$}] (u1) at (0,-2.5) {};
						\node[label=right:\tiny{${<p_1p_2^3>}$}] (u2) at (3,0) {};
						\node[label=right:\tiny{${<p_1^2p_2^3>}$}] (u3) at (3,1) {};
						\node[label=left:\tiny{${<p_1^4p_2^2>}$}] (u4) at (1,-1.5) {};
						\node[label=below:\tiny{${<p_1^4p_2>}$}] (u5) at (6.5,-2) {};
						\node[label=right:\tiny{${<p_1^4>}$}] (u6) at (7,1.5) {};
						\node[label=left:\tiny{${<p_1^4>}$}] (u7) at (-3,-2) {};
						\node[label=right:\tiny{${<p_1^3p_2^3>}$}] (u9) at (3,3) {};
						
						\node[label=below:\tiny{$e_3$}] (v3) at (1,-2) {};
						\node[label=right:\tiny{$e_2$}] (v2) at (2,.5) {};
						\node[label=above:\tiny{$e_1$}] (v1) at (2,1) {};
						\node[label=below:\tiny{$e_4$}] (v4) at (4.5,-1) {};
						\node[label=above:\tiny{$e_5$}] (v5) at (4,2) {};
						\node[label=above:\tiny{$e_6$}] (v6) at (4,-4.5) {};
						\node[label=left:\tiny{$e_7$}] (v7) at (3,-3) {};
						\node[label=above:\tiny{$e_8$}] (v8) at (5,0) {};
						\node[label=below:\tiny{$e_9$}] (v9) at (1,-5) {};
						\node[label=above:\tiny{$e_9$}] (v10) at (0,4) {};
						\node[label=above:\tiny{$e_{10}$}] (v13) at (0,3) {};
						\node[label=above:\tiny{$e_{11}$}] (v14) at (0,2) {};
						\node[label=below:\tiny{$e_{12}$}] (v11) at (0,-5) {};
						\node[label=above:\tiny{$e_{12}$}] (v12) at (1,4) {};
						
						
						
						\draw (u4) -- (v1);
						\draw (v1) -- (u3);
						\draw (u4) -- (v2);
						\draw (v2) -- (u2);
						\draw (u4) -- (v3);
						\draw (v3) -- (u1);
						\draw (u5) -- (v4);
						\draw (v4) -- (u3);
						\draw (u6) -- (v5);
						\draw (v5) -- (u3);
						\draw (v6) -- (u1);
						\draw (u6) -- (v8);
						\draw (u3) -- (v8);
						\draw (u1) -- (v9);
						\draw (u7) -- (v10);
						\draw (u7) -- (v13);
						\draw (u7) -- (v14);
						\draw (u9) -- (v13);
						\draw (u9) -- (v14);
						\draw (u5) -- (v6);
						\draw (u5)-- (v7);
						\draw (v7) -- (u2);
						\draw (u4) to[out=180, in=180](v11);
						\draw (u9) -- (v12);
						\draw[->,] (0,4) -- (2,4) [right] {};
						\draw[->,] (2,-5) -- (-1,-5) [right] {};
						\draw[->,] (7,1) -- (7,0) [right] {};
						\draw[->,] (-3,0) -- (-3,1) [right] {};
					\end{tikzpicture}}
				\hspace{.2cm}
				\subfloat[$\mathcal{I}(\tilde{\Gamma}_\mathcal{H}(\mathbb{Z}_{p_1^5p_2^3}))$ on projective plane]{				\begin{tikzpicture}[scale=.65, every node/.style={circle, fill=black, minimum size=4pt, inner sep=0pt}]
						\draw (-3,-5) rectangle (7,4); 
						
						\node[label=left:\tiny{${<p_2^3>}$}] (u1) at (0,-2.5) {};
						\node[label=below:\tiny{${<p_1p_2^3>}$}] (u2) at (3,0) {};
						\node[label=right:\tiny{${<p_1^2p_2^3>}$}] (u3) at (3,1) {};
						\node[label=left:\tiny{${<p_1^5p_2^2>}$}] (u4) at (1,-1.5) {};
						\node[label=below:\tiny{${<p_1^5p_2>}$}] (u5) at (7,-2) {};
						\node[label=right:\tiny{${<p_1^5>}$}] (u6) at (7,1.5) {};
						\node[label=left:\tiny{${<p_1^5>}$}] (u7) at (-3,-2) {};
						\node[label=left:\tiny{${<p_1^5p_2>}$}] (u8) at (-3,1.5) {};
						\node[label=right:\tiny{${<p_1^3p_2^3>}$}] (u9) at (3,3) {};
						\node[label=above:\tiny{${<p_1^4p_2^3>}$}] (u10) at (4,4) {};
						\node[label=below:\tiny{${<p_1^4p_2^3>}$}] (u11) at (0,-5) {};
						
						\node[label=below:\tiny{$e_3$}] (v3) at (1,-2) {};
						\node[label=right:\tiny{$e_2$}] (v2) at (2,.5) {};
						\node[label=above:\tiny{$e_1$}] (v1) at (2,1) {};
						\node[label=below:\tiny{$e_4$}] (v4) at (4,0) {};
						\node[label=above:\tiny{$e_5$}] (v5) at (4,2) {};
						\node[label=left:\tiny{$e_6$}] (v6) at (-2,-4) {};
						\node[label=above:\tiny{$e_7$}] (v7) at (5,-2) {};
						\node[label=above:\tiny{$e_8$}] (v8) at (5,0) {};
						\node[label=below:\tiny{$e_9$}] (v9) at (4,-5) {};
						\node[label=above:\tiny{$e_9$}] (v10) at (0,4) {};
						\node[label=above:\tiny{$e_{10}$}] (v13) at (0,3) {};
						\node[label=above:\tiny{$e_{11}$}] (v14) at (0,2) {};
						\node[label=below:\tiny{$e_{12}$}] (v11) at (1,-5) {};
						\node[label=above:\tiny{$e_{12}$}] (v12) at (1,4) {};
						\node[label=right:\tiny{$e_{13}$}] (v15) at (5,3) {};
						\node[label=left:\tiny{$e_{14}$}] (v16) at (5,-3) {};
						\node[label=right:\tiny{$e_{15}$}] (v17) at (3,-2) {};
						
						
						
						\draw (u4) -- (v1);
						\draw (v1) -- (u3);
						\draw (u4) -- (v2);
						\draw (v2) -- (u2);
						\draw (u4) -- (v3);
						\draw (v3) -- (u1);
						\draw (u5) -- (v4);
						\draw (v4) -- (u3);
						\draw (u6) -- (v5);
						\draw (v5) -- (u3);
						\draw (v6) -- (u1);
						\draw (u6) -- (v8);
						\draw (u3) -- (v8);
						\draw (u1) -- (v9);
						\draw (u7) -- (v10);
						\draw (u7) -- (v13);
						\draw (u7) -- (v14);
						\draw (u9) -- (v13);
						\draw (u9) -- (v14);
						\draw (u10) -- (v15);
						\draw (u6) -- (v15);
						\draw (u11) -- (v16);
						\draw (u5) -- (v16);
						\draw (u11) -- (v17);
						
						\draw (u5) -- (v7);
						\draw (v7) -- (u2);
						\draw (u8) to[out=0, in=90](v6);
						\draw (u4) to[out=0, in=90](v11);
						\draw (u9) -- (v12);
						\draw (u4) to[out=0, in=100](v17);
						\draw[->,] (0,4) -- (2,4) [right] {};
						\draw[->,] (2,-5) -- (-1,-5) [right] {};
						\draw[->,] (7,1) -- (7,0) [right] {};
						\draw[->,] (-3,0) -- (-3,1) [right] {};
					\end{tikzpicture}}
				\hspace{.2cm}
			\subfloat[$\mathcal{I}(\tilde{\Gamma}_\mathcal{H}(\mathbb{Z}_{p_1^4p_2^4}))$ on projective plane]{		\begin{tikzpicture}[scale=.65, every node/.style={circle, fill=black, minimum size=3pt, inner sep=0pt}]
						\draw (-4,-5) rectangle (7,4); 
						
						\node[label=left:\tiny{${<p_2^4>}$}] (u1) at (0,-2.5) {};
						\node[label=right:\tiny{${<p_1p_2^4>}$}] (u2) at (7,-1) {};
						\node[label=right:\tiny{${<p_1^2p_2^4>}$}] (u3) at (3,1) {};
						\node[label=right:\tiny{${<p_1^4p_2^2>}$}] (u4) at (1,-1.5) {};
						\node[label=below:\tiny{${<p_1^4p_2>}$}] (u5) at (7,-2) {};
						\node[label=right:\tiny{${<p_1^4>}$}] (u6) at (7,1.5) {};
						\node[label=left:\tiny{${<p_1^4>}$}] (u7) at (-4,-1) {};
						\node[label=left:\tiny{${<p_1^4p_2>}$}] (u8) at (-4,1.5) {};
						\node[label=right:\tiny{${<p_1^3p_2^4>}$}] (u9) at (3,3) {};
						\node[label=right:\tiny{${<p_1^4p_2^3>}$}] (u10) at (1,2) {};
						\node[label=left:\tiny{${<p_1p_2^4>}$}] (u11) at (-4,-2) {};
						
						\node[label=below:\tiny{$e_9$}] (v3) at (1,-2) {};
						
						\node[label=above:\tiny{$e_2$}] (v19) at (-2,-1) {};
						\node[label=left:\tiny{$e_{11}$}] (v1) at (2,1) {};
						5) {};
						\node[label=above:\tiny{$e_3$}] (v5) at (5,1.5) {};
						\node[label=right:\tiny{$e_5$}] (v6) at (0,-4) {};
						\node[label=above:\tiny{$e_6$}] (v7) at (4,-2) {};
						\node[label=below:\tiny{$e_1$}] (v9) at (5,-5) {};
						\node[label=above:\tiny{$e_1$}] (v10) at (0,4) {};
						\node[label=above:\tiny{$e_4$}] (v13) at (0,3) {};
						\node[label=below:\tiny{$e_{12}$}] (v11) at (3,-5) {};
						\node[label=above:\tiny{$e_{12}$}] (v12) at (1,4) {};
						\node[label=left:\tiny{$e_{13}$}] (v15) at (0,0) {};
						\node[label=below:\tiny{$e_{14}$}] (v16) at (-1,.8) {};
						\node[label=left:\tiny{$e_{15}$}] (v17) at (2,1.5) {};
						\node[label=right:\tiny{$e_{16}$}] (v18) at (3,2.5) {};
						\node[label=below:\tiny{$e_7$}] (v20) at (0,-5) {};
						\node[label=above:\tiny{$e_8$}] (v21) at (3,4) {};
						\node[label=below:\tiny{$e_8$}] (v22) at (1,-5) {};
						\node[label=above:\tiny{$e_{10}$}] (v23) at (2.5,0) {};
						\node[label=above:\tiny{$e_7$}] (v24) at (5,4) {};
						
						
						
						\draw (u4) -- (v1);
						\draw (v1) -- (u3);
						\draw (u4) -- (v3);
						\draw (v3) -- (u1);
						\draw (u6) -- (v5);
						\draw (v5) -- (u3);
						\draw (v6) -- (u1);
						\draw (u1) -- (v9);
						\draw (u7) -- (v10);
						\draw (u7) -- (v13);
						\draw (u9) -- (v13);
						\draw (u1) -- (v15);
						\draw (u10) -- (v15);
						\draw (u11) -- (v16);
						\draw (u10) -- (v16);
						\draw (u3) -- (v17);
						\draw (u10) -- (v17);
						\draw (u9) -- (v18);
						\draw (u10) -- (v18);
						\draw (u7) -- (v19);
						\draw (u11) -- (v19);
						\draw (u3) -- (v24);
						\draw (u5) -- (v20);
						\draw (u9) -- (v21);
						\draw (u5) -- (v22);
						\draw (u2) -- (v23);
						\draw (u4) -- (v23);
						
						\draw (u5) to[out=200, in=0](v7);
						\draw (v7) to[out=0, in=200](u2);
						\draw (u8) to[out=300, in=180](v6);
						\draw (u4) to[out=330, in=90](v11);
						\draw (u9) -- (v12);
						\draw[->,] (0,4) -- (2,4) [right] {};
						\draw[->,] (2,-5) -- (-1,-5) [right] {};
						\draw[->,] (7,1) -- (7,0) [right] {};
						\draw[->,] (-4,0) -- (-4,1) [right] {};
					\end{tikzpicture}}

				\caption{}
				\label{proj1}
			\end{figure}

		\begin{figure}[htbp!]
			\centering
			\subfloat[$\mathcal{I}(\tilde{\Gamma}_\mathcal{H}(\mathbb{Z}_{p_1^4p_2p_3}))$ on projective plane]{
			\begin{tikzpicture}[scale=1.2, every node/.style={circle, fill=black, minimum size=4pt, inner sep=0pt}]
						\draw (-5,.5) rectangle (3,7); 
						
						\node[label=above:\tiny{$<p_1^4p_2>$}] (u1) at (-1,6.8) {};
						\node[label=left:\tiny{$<p_1^4p_3>$}] (u2) at (-5,3.5) {};
						\node[label=left:\tiny{$<p_2p_3>$}] (u3) at (1.5,3) {};
						\node[label=above:\tiny{$<p_1p_2p_3>$}] (u4) at (-2,7) {};
						\node[label=above:\tiny{$<p_3>$}] (u5) at (-1.7,6.5) {};
						\node[label=above left:\tiny{$<p_1p_3>$}] (u6) at (-2,5) {};
						\node[label=below:\tiny{$<p_2>$}] (u7) at (-4,2) {};
						\node[label=below:\tiny{$<p_1p_2>$}] (u8) at (-3,2) {};
						\node[label=above left:\tiny{$<p_1^4>$}] (u9) at (0,3) {};
						\node[label=above:\tiny{$<p_1^2p_2p_3>$}] (u10) at (-.3,5) {};
						\node[label=right:\tiny{$<p_1^2p_3>$}] (u11) at (1,6.5) {};
						\node[label=right:\tiny{$<p_1^2p_2>$}] (u12) at (-3,3.8) {};
						\node[label=left:\tiny{$<p_1^4p_3>$}] (u14) at (3,3.5) {};
						\node[label=below:\tiny{$<p_1p_2p_3>$}] (u15) at (0,.5) {};
						\node[label=below:\tiny{$<p_1^3p_2p_3>$}] (u16) at (-3,.5) {};
						\node[label=above:\tiny{$<p_1^3p_2p_3>$}] (u17) at (2,7) {};
						\node[label=right:\tiny{$<p_1^3p_3>$}] (u18) at (1.5,6) {};
						\node[label=right:\tiny{$<p_1^3p_2>$}] (u19) at (-1.8,3) {};
						
						\node[label=below:\tiny{$e_1$}] (v1) at (1,4) {};
						\node[label=below:\tiny{$e_2$}] (v2) at (-3,5) {};
						\node[label=above:\tiny{$e_{16}$}] (v3) at (1,2) {};
						\node[label=right:\tiny{$e_{15}$}] (v4) at (1,1) {};
						\node[label=left:\tiny{$e_5$}] (v5) at (-2,6) {};
						\node[label=below:\tiny{$e_6$}] (v6) at (-1.5,5) {};
						\node[label=left:\tiny{$e_9$}] (v7) at (-3,3) {};
						\node[label=right:\tiny{$e_{10}$}] (v8) at (-2.5,3) {};
						\node[label=left:\tiny{$e_3$}] (v9) at (-1,4.5) {};
						\node[label=below:\tiny{$e_7$}] (v10) at (0,6.5) {};
						\node[label=right:\tiny{$e_{11}$}] (v11) at (-3.5,3.5) {};
						\node[label=right:\tiny{$e_{14}$}] (v12) at (0,4.5) {};
						\node[label=right:\tiny{$e_4$}] (v13) at (2,5) {};
						\node[label=below:\tiny{$e_8$}] (v14) at (.5,6) {};
						\node[label=right:\tiny{$e_{12}$}] (v15) at (-2.5,3.4) {};
						\node[label=right:\tiny{$e_{13}$}] (v16) at (-1,1) {};
						
						\draw (u1) -- (v1);
						\draw (u14) -- (v1);
						\draw (u3) -- (v1);
						\draw (u1) -- (v2);
						\draw (u2) -- (v2);
						\draw (u4) -- (v2);
						\draw (u9) -- (v3);
						\draw (u3) to[out=0, in=10](v3);
						\draw (u15) -- (v4);
						\draw (u9) -- (v4);
						\draw (u1) -- (v5);
						\draw (u5) -- (v5);
						\draw (u1) -- (v6);
						\draw (u6) -- (v6);
						\draw (u2) -- (v7);
						\draw (u7) -- (v7);
						\draw (u2) -- (v8);
						\draw (u8) -- (v8);
						\draw (u1) -- (v9);
						\draw (u2) -- (v9);
						\draw (u10) -- (v9);
						\draw (u1) -- (v10);
						\draw (u11) -- (v10);
						\draw (u2) -- (v11);
						\draw (u12) -- (v11);
						\draw (u9) -- (v12);
						\draw (u10) -- (v12);
						\draw (u14)  -- (v13);
						\draw (u17) -- (v13);
						\draw (u1)  -- (v13);
						\draw (u18) -- (v14);
						\draw (u1)  -- (v14);
						\draw (u19) -- (v15);
						\draw (u2)  -- (v15);
						\draw (u9) -- (v16);
						\draw (u16)  -- (v16);
						
						\draw[->,] (0,7) -- (1,7) [right] {};
						\draw[->,] (2,.5) -- (-1,.5) [right] {};
						\draw[->,] (3,3) -- (3,2) [right] {};
						\draw[->,] (-5,2) -- (-5,3) [right] {};
				\end{tikzpicture}}
			\hspace{.2cm}
			\subfloat[$\mathcal{I}(\tilde{\Gamma}_\mathcal{H}(\mathbb{Z}_{p_1^5p_2p_3}))$ on projective plane]{
				\begin{tikzpicture}[scale=1.2, every node/.style={circle, fill=black, minimum size=4pt, inner sep=0pt}]
						\draw (-5,.5) rectangle (3,7); 
						
						\node[label=above:\tiny{$<p_1^5p_2>$}] (u1) at (-1,6.8) {};
						\node[label=left:\tiny{$<p_1^5p_3>$}] (u2) at (-5,3.5) {};
						\node[label=left:\tiny{$<p_2p_3>$}] (u3) at (1.5,3) {};
						\node[label=above:\tiny{$<p_1p_2p_3>$}] (u4) at (-2,7) {};
						\node[label=right:\tiny{$<p_3>$}] (u5) at (-1.7,6.5) {};
						\node[label=above left:\tiny{$<p_1p_3>$}] (u6) at (-2,5) {};
						\node[label=below:\tiny{$<p_2>$}] (u7) at (-4,2) {};
						\node[label=below:\tiny{$<p_1p_2>$}] (u8) at (-3,2) {};
						\node[label=above left:\tiny{$<p_1^5>$}] (u9) at (0,3) {};
						\node[label=above:\tiny{$<p_1^2p_2p_3>$}] (u10) at (-.3,5) {};
						\node[label=right:\tiny{$<p_1^2p_3>$}] (u11) at (1,6.5) {};
						\node[label=right:\tiny{$<p_1^2p_2>$}] (u12) at (-3,3.8) {};
						\node[label=right:\tiny{$<p_1^5p_3>$}] (u14) at (3,3.5) {};
						\node[label=below:\tiny{$<p_1p_2p_3>$}] (u15) at (0,.5) {};
						\node[label=below:\tiny{$<p_1^3p_2p_3>$}] (u16) at (-3,.5) {};
						\node[label=above:\tiny{$<p_1^3p_2p_3>$}] (u17) at (2,7) {};
						\node[label=right:\tiny{$<p_1^3p_3>$}] (u18) at (1.5,6) {};
						\node[label=right:\tiny{$<p_1^3p_2>$}] (u19) at (-1.8,3) {};
						\node[label=left:\tiny{$<p_1^4p_2>$}] (u20) at (-4,2.5) {};
						\node[label=right:\tiny{$<p_1^4p_2>$}] (u21) at (1.4,4.5) {};
						\node[label=right:\tiny{$<p_1^5p_2p_3>$}] (u22) at (1.8,3) {};
						
						\node[label=below:\tiny{$e_1$}] (v1) at (1,4) {};
						\node[label=below:\tiny{$e_2$}] (v2) at (-3,5) {};
						\node[label=above:\tiny{$e_{13}$}] (v3) at (1,2) {};
						\node[label=right:\tiny{$e_{14}$}] (v4) at (1,1) {};
						\node[label=left:\tiny{$e_5$}] (v5) at (-2,6) {};
						\node[label=below:\tiny{$e_6$}] (v6) at (-1.5,5) {};
						\node[label=left:\tiny{$e_9$}] (v7) at (-3,3) {};
						\node[label=right:\tiny{$e_{10}$}] (v8) at (-2.5,3) {};
						\node[label=left:\tiny{$e_3$}] (v9) at (-1,4.5) {};
						\node[label=below:\tiny{$e_7$}] (v10) at (0,6.5) {};
						\node[label=right:\tiny{$e_{11}$}] (v11) at (-3.5,3.5) {};
						\node[label=right:\tiny{$e_{15}$}] (v12) at (0,4.5) {};
						\node[label=right:\tiny{$e_4$}] (v13) at (2,5) {};
						\node[label=below:\tiny{$e_8$}] (v14) at (.5,6) {};
						\node[label=right:\tiny{$e_{12}$}] (v15) at (-2.5,3.4) {};
						\node[label=right:\tiny{$e_{16}$}] (v16) at (-1,1) {};
						\node[label=left:\tiny{$e_{17}$}] (v17) at (-4,3) {};
						\node[label=left:\tiny{$e_{18}$}] (v18) at (.8,5) {};
						\node[label=below:\tiny{$e_{19}$}] (v19) at (2,.5) {};
						\node[label=above:\tiny{$e_{19}$}] (v20) at (-3,7) {};
						\node[label=right:\tiny{$e_{14}$}] (v21) at (1.5,1.2) {};
						
						\draw (u1) -- (v1);
						\draw (u14) -- (v1);
						\draw (u3) -- (v1);
						\draw (u1) -- (v2);
						\draw (u2) -- (v2);
						\draw (u4) -- (v2);
						\draw (u9) -- (v3);
						\draw (u3) -- (v3);
						\draw (u15) -- (v4);
						\draw (u9) -- (v4);
						\draw (u1) -- (v5);
						\draw (u5) -- (v5);
						\draw (u1) -- (v6);
						\draw (u6) -- (v6);
						\draw (u2) -- (v7);
						\draw (u7) -- (v7);
						\draw (u2) -- (v8);
						\draw (u8) -- (v8);
						\draw (u1) -- (v9);
						\draw (u2) -- (v9);
						\draw (u10) -- (v9);
						\draw (u1) -- (v10);
						\draw (u11) -- (v10);
						\draw (u2) -- (v11);
						\draw (u12) -- (v11);
						\draw (u9) -- (v12);
						\draw (u10) -- (v12);
						\draw (u14)  -- (v13);
						\draw (u17) -- (v13);
						\draw (u1)  -- (v13);
						\draw (u18) -- (v14);
						\draw (u1)  -- (v14);
						\draw (u19) -- (v15);
						\draw (u2)  -- (v15);
						\draw (u9) -- (v16);
						\draw (u16)  -- (v16);
						\draw (u2) -- (v17);
						\draw (u20) -- (v17);
						\draw (u1) -- (v18);
						\draw (u21)  -- (v18);
						\draw (u14) -- (v19);
						\draw (u22)  -- (v19);
						\draw (u1)  --  (v20);
						\draw (u9) -- (v21);
						\draw (u22)  -- (v21);
						
						\draw[->,] (0,7) -- (1,7) [right] {};
						\draw[->,] (2,.5) -- (-1,.5) [right] {};
						\draw[->,] (3,3) -- (3,2) [right] {};
						\draw[->,] (-5,2) -- (-5,3) [right] {};
				\end{tikzpicture}}
			\caption{}
			\label{proj2}
			\end{figure}
			\noindent\textbf{Case 2.} Suppose $\omega(n)  = 3$.\\
		\textbf{Subcase 2.1.} If $n=p_1p_2p_3$ or $n=p_1^2p_2p_3$, then, by Theorem \ref{planarity}, $\tilde{\Gamma}_\mathcal{H}(\mathbb{Z}_n)$ is planar and hence, $\tilde{g}(\tilde{\Gamma}_\mathcal{H}(\mathbb{Z}_n)) \neq 1$.\\
			\textbf{Subcase 2.2.}  If $n \neq p_1p_2p_3$ or $n \neq p_1^2p_2p_3$, then consider the following cases:\\
			\textbf{Subsubcase 2.2.(a)} If $n$ divides $p_1^4p_2p_3$, then  $\mathcal{I}(\tilde{\Gamma}_\mathcal{H}(\mathbb{Z}_n))$ is isomorphic to a subgraph of $\mathcal{I}(\tilde{\Gamma}_\mathcal{H}(\mathbb{Z}_{p_1^4p_2p_3}))$ and   Figure \ref{proj2}(a) depicts an embedding of $\mathcal{I}(\tilde{\Gamma}_\mathcal{H}(\mathbb{Z}_{p_1^4p_2p_3}))$ on a projective plane. Hence, in this case,  $\tilde{g}(\tilde{\Gamma}_\mathcal{H}(\mathbb{Z}_n)) = 1$. \\
			\textbf{Subsubcase 2.2.(b)} If $n$ is divisible by $p_1^5p_2p_3$, then $\mathcal{I}(\tilde{\Gamma}_\mathcal{H}(\mathbb{Z}_n))$ contains a subgraph isomorphic to $\mathcal{I}(\tilde{\Gamma}_\mathcal{H}(\mathbb{Z}_{p_1^5p_2p_3}))$ and Figure \ref{proj2}(b) depicts that $\mathcal{I}(\tilde{\Gamma}_\mathcal{H}(\mathbb{Z}_{p_1^5p_2p_3}))$ is not embeddable on a projective plane. Hence, in this case,  $\tilde{g}(\tilde{\Gamma}_\mathcal{H}(\mathbb{Z}_n)) \neq 1$.
			\\ 
			\textbf{Subsubcase 2.2.(c)} If $n$ contain atleast two prime divisors with power greater than  or equal to 2, then  $\mathcal{I}(\tilde{\Gamma}_\mathcal{H}(\mathbb{Z}_n))$ contains a subgraph isomorphic to $\mathcal{I}(\tilde{\Gamma}_\mathcal{H}(\mathbb{Z}_{p_1^2p_2^2p_3}))$ and Figure \ref{proj3}(a) depicts that  $\mathcal{I}(\tilde{\Gamma}_\mathcal{H}(\mathbb{Z}_{p_1^2p_2^2p_3}))$  is not embeddable on a projective plane. Hence, in this case,   $\tilde{g}(\tilde{\Gamma}_\mathcal{H}(\mathbb{Z}_n)) \neq 1$.\\
			\noindent\textbf{Case 3.} Suppose $\omega(n) \geq 4$.\\
			Observe that $\mathcal{I}(\tilde{\Gamma}_\mathcal{H}(\mathbb{Z}_n))$ contains a subgraph isomorphic to $\mathcal{I}(\tilde{\Gamma}_\mathcal{H}(\mathbb{Z}_{p_1p_2p_3p_4}))$ and Figure \ref{proj3}(b) depicts that  $\mathcal{I}(\tilde{\Gamma}_\mathcal{H}(\mathbb{Z}_{p_1p_2p_3p_4}))$ is not embeddable on a projective plane  and hence, in this case,   $\tilde{g}(\tilde{\Gamma}_\mathcal{H}(\mathbb{Z}_n)) \neq 1$.
			\begin{figure}[htbp!]
					\centering
				\subfloat[$\mathcal{I}(\tilde{\Gamma}_\mathcal{H}(\mathbb{Z}_{p_1^2p_2^2p_3}))$  on a projective plane]{\begin{tikzpicture}[scale=1, every node/.style={circle, fill=black, minimum size=3pt, inner sep=0pt}]
							\draw (-4,0) rectangle (4,8); 
							
							\node[label=above:\tiny{$<p_1^2p_2^2>$}] (u1) at (-1,8) {};
							\node[label=right:\tiny{$<p_1^2p_2p_3>$}] (u2) at (-3,3.5) {};
							\node[label=right:\tiny{$<p_1p_2^2p_3>$}] (u3) at (4,5) {};
							\node[label=left:\tiny{$<p_2^2p_3>$}] (u4) at (-1,5) {};
							\node[label=left:\tiny{$<p_1^2p_3>$}] (u5) at (1,6.2) {};
							\node[label=right:\tiny{$<p_1p_2p_3>$}] (u6) at (0,3.5) {};
							\node[label=right:\tiny{$<p_1p_3>$}] (u7) at (.2,2.3) {};
							\node[label=right:\tiny{$<p_2p_3>$}] (u8) at (.2,1) {};
							\node[label=left:\tiny{$<p_3>$}] (u9) at (-2,6.5) {};
						
							\node[label=right:\tiny{$<p_1p_2^2>$}] (u11) at (1.4,4) {};
							\node[label=left:\tiny{$<p_1^2p_2>$}] (u12) at (3,1) {};
							\node[label=right:\tiny{$<p_1^2>$}] (u13) at (-2.5,6.2) {};
							\node[label=below:\tiny{$<p_1^2p_2^2>$}] (u14) at (-1,0) {};
							\node[label=left:\tiny{$<p_1p_2^2p_3>$}] (u15) at (-4,5) {};
							\node[label=right:\tiny{$<p_2^2>$}] (u16) at (-1.5,2.4) {};
							
							\node[label=left:\tiny{$e_1$}] (v1) at (-3.5,3.5) {};
							\node[label=below right:\tiny{$e_2$}] (v2) at (-3,4.5) {};
							\node[label=above:\tiny{$e_3$}] (v3) at (3,6.4) {};
							\node[label=below:\tiny{$e_4$}] (v4) at (-1,6.5) {};
							\node[label=right:\tiny{$e_5$}] (v5) at (-0.5,3) {};
							\node[label=below right:\tiny{$e_6$}] (v6) at (-0.5,2) {};
							\node[label=below:\tiny{$e_7$}] (v7) at (0,1) {};
							\node[label=right:\tiny{$e_8$}] (v8) at (-1.5,6.5) {};
							\node[label=below:\tiny{$e_9$}] (v9) at (0,4.5) {};
							\node[label=left:\tiny{$e_{10}$}] (v10) at (-3,2) {};
							\node[label=left:\tiny{$e_{11}$}] (v11) at (3,3.5) {};
							\node[label=right:\tiny{$e_{12}$}] (v12) at (-3.5,5.5) {};
							\node[label=right:\tiny{$e_{13}$}] (v13) at (0,5) {};
							\node[label=right:\tiny{$e_{14}$}] (v14) at (2,5) {};
							\node[label=left:\tiny{$e_{15}$}] (v15) at (0,5.5) {};

							
							\draw (u1) to[out=200, in=70] (v1);
							\draw (u2) -- (v1);
							\draw (u15) -- (v1);
							\draw (u14) -- (v2);
							\draw (u2) -- (v2);
							\draw (u4) -- (v2);
							\draw (u1) -- (v3);
							\draw (u3) -- (v3);
							\draw (u5) -- (v3);
							\draw (u1) -- (v4);
							\draw (u4) -- (v4);
							\draw (u5) -- (v4);
							\draw (u6) -- (v5);
							\draw (u14) -- (v5);
							\draw (u7) -- (v6);
							\draw (u14) -- (v6);
							\draw (u8) -- (v7);
							\draw (u14) -- (v7);
							\draw (u1) -- (v8);
							\draw (u9) -- (v8);
							\draw (u2) -- (v9);
							\draw (u11) -- (v9);
							\draw (u2) -- (v10);
							\draw (u16) -- (v10);
							\draw (u3) -- (v11);
							\draw (u12) -- (v11);
							\draw (u15) -- (v12);
							\draw (u13) -- (v12);
							\draw (u11) -- (v13);
							\draw (u4) -- (v13);
							\draw (u5) -- (v14);
							\draw (u11) -- (v14);
							\draw (u5) -- (v15);
							\draw (u16) -- (v15);
							
							\draw[->,] (0,8) -- (2,8) [right] {};
							\draw[->,] (2,0) -- (-2,0) [right] {};
							\draw[->,] (-4,4) -- (-4,6) [right] {};
							\draw[->,] (4,5) -- (4,4) [right] {};

					\end{tikzpicture}}
		\hspace{.2cm}
				\subfloat[$\mathcal{I}(\tilde{\Gamma}_{\mathcal{H}}(\mathbb{Z}_{p_1p_2p_3p_4}))$ on a projective plane]{
					\begin{tikzpicture}[scale=1.1, every node/.style={circle, fill=black, minimum size=3pt, inner sep=0pt}]
							\draw (-3,-3) rectangle (5,4); 
							
							\node[label=above:\tiny{${<p_1p_2p_3>}$}] (u1) at (0,4) {};
							\node[label=right:\tiny{${<p_1p_2p_4>}$}] (u2) at (1,3) {};
							\node[label=right:\tiny{${<p_1p_3p_4>}$}] (u3) at (-1,2.5) {};
							\node[label=right:\tiny{${<p_2p_3p_4>}$}] (u4) at (-2,2) {};
							\node[label=above:\tiny{${<p_3p_4>}$}] (u5) at (2,4) {};
							\node[label=left:\tiny{${<p_2p_4>}$}] (u6) at (-.8,3) {};
							\node[label=below:\tiny{${<p_1p_2p_3>}$}] (u7) at (0,-3) {};
							\node[label=left:\tiny{${<p_1p_4>}$}] (u8) at (0,-1) {};
							\node[label=left:\tiny{${<p_4>}$}] (u9) at (1.5,-2) {};
							\node[label=right:\tiny{${<p_2p_3>}$}] (u10) at (-2,.5) {};
							\node[label=right:\tiny{${<p_1p_3>}$}] (u11) at (3,0) {};
							\node[label=below:\tiny{${<p_3>}$}] (u12) at (2.5,1) {};
							\node[label=below:\tiny{${<p_1p_2>}$}] (u13) at (-2.5,2.5) {};
							\node[label=left:\tiny{${<p_2>}$}] (u14) at (-2.5,3) {};
							\node[label=left:\tiny{${<p_3>}$}] (u15) at (-2.2,1) {};
							\node[label=below:\tiny{${<p_3p_4>}$}] (u16) at (-2,-3) {};

							\node[label=below:\tiny{$e_1$}] (v1) at (2,1) {};
							\node[label=right:\tiny{$e_2$}] (v2) at (1.5,3.5) {};
							\node[label=below:\tiny{$e_3$}] (v3) at (-.2,3) {};
							\node[label=right:\tiny{$e_4$}] (v4) at (2,0) {};
							\node[label=below:\tiny{$e_5$}] (v5) at (1,-2.5) {};
							\node[label=left:\tiny{$e_6$}] (v6) at (-3,-.2) {};
							\node[label=right:\tiny{$e_6$}] (v7) at (5,2) {};
							\node[label=right:\tiny{$e_7$}] (v8) at (3.5,2) {};
							\node[label=right:\tiny{$e_8$}] (v9) at (2.5,2) {};
							\node[label=above:\tiny{$e_9$}] (v10) at (-2,2.5) {};
							\node[label=left:\tiny{$e_{10}$}] (v11) at (-2,3.5) {};
							\node[label=right:\tiny{$e_{11}$}] (v12) at (-2.5,1.5) {};
							\node[label=left:\tiny{$e_{12}$}] (v13) at (-3,2.5) {};
							\node[label=below:\tiny{$e_{13}$}] (v14) at (3,-3) {};
							\node[label=above:\tiny{$e_{14}$}] (v16) at (0,0) {};
							\node[label=above:\tiny{$e_{13}$}] (v17) at (-2,4) {};
							\node[label=right:\tiny{$e_{12}$}] (v18) at (5,0) {};

							
							
							
							\draw (u1) -- (v1);
							\draw (u2) -- (v1);
							\draw (u3) -- (v1);
							\draw (u4) -- (v1);
							\draw (u1) -- (v2);
							\draw (u2) -- (v2);
							\draw (u5) -- (v2);
							\draw (u1) -- (v3);
							\draw (u3) -- (v3);
							\draw (u6) -- (v3);
							\draw (u7) -- (v4);
							\draw (u4) -- (v4);
							\draw (u8) -- (v4);
							\draw (u9) -- (v5);
							\draw (u7) -- (v5);
							\draw (u2) -- (v7);
							\draw (u3) -- (v6);
							\draw (u10) -- (v6);
							\draw (u2) -- (v8);
							\draw (u11) -- (v8);
							\draw (u2) -- (v9);
							\draw (u12) -- (v9);
							\draw (u3) -- (v10);
							\draw (u4) -- (v10);
							\draw (u13) -- (v10);
							\draw (u3) -- (v11);
							\draw (u14) -- (v11);
							\draw (u4) -- (v12);
							\draw (u15) -- (v12);
							\draw (u13) -- (v13);
							\draw (u11) -- (v14);
							\draw (u6) -- (v17);
							\draw (u8) -- (v16);
							\draw (u10) -- (v16);
							\draw (u16) -- (v18);


							\draw (u4) to[out=300, in=250](v8);
							\draw[->,] (-2,4) -- (-1,4) [right] {};
							\draw[->,] (0,-3) -- (-1,-3) [right] {};
							\draw[->,] (5,2) -- (5,1) [right] {};
							\draw[->,] (-3,0) -- (-3,.5) [right] {};
					\end{tikzpicture}}
				\caption{}
				\label{proj3}
			\end{figure}
		\end{proof}
		\begin{corollary}
			$\tilde{\Gamma}_\mathcal{H}(\mathbb{Z}_n)$ is projective if and only if $n=p_1^{\alpha_1}p_2^{\alpha}$ with $\alpha=1,2$ or $n=p_1^{\beta}p_2^3$ with  $\beta=3,4$ or $n=p_1^{\gamma}p_2p_3$ with $1 \leq \gamma \leq 4$, where $p_1,p_2,p_3$ are distinct primes and $\alpha_1$ is a positive integer.
		\end{corollary}

\end{document}